%% file: slow_conv.tex
\documentclass[11pt]{amsart}

\usepackage{color,graphicx,enumerate,wrapfig,amssymb}

\usepackage{subfigure}

\usepackage{eucal}

\usepackage{setspace}

\usepackage{hyperref}

\usepackage{graphicx}
\usepackage{transparent}

\usepackage{enumitem}
\newcommand\bulitem[1]{\item{\bfseries #1 }} 

\newcommand\indentdisplays[1]{%
     \everydisplay{\addtolength\displayindent{#1}%
     \addtolength\displaywidth{-#1}}}  

\newcommand\numberthis{\addtocounter{equation}{1}\tag{\theequation}} 

\hypersetup{
    bookmarks=true,         
    unicode=true,          
    pdftoolbar=true,        
    pdfmenubar=true,        
    pdffitwindow=false,     
    pdfstartview={FitH},    
    pdftitle={My title},    
    pdfauthor={Author},     
    pdfsubject={Subject},   
    pdfcreator={Creator},   
    pdfproducer={Producer}, 
    pdfkeywords={keyword1} {key2} {key3}, 
    pdfnewwindow=true,      
    colorlinks=true,       
    linkcolor=blue,          
    citecolor=blue,        
    filecolor=magenta,      
    urlcolor=cyan           
}

\renewenvironment{proof}[1][\proofname ]{{\noindent \bfseries #1. }}{\qed \bigskip } 

\newcommand{\R}{{\mathbb R}}

\newcommand{\Z}{{\mathbb Z}}
\newcommand{\N}{{\mathbb N}}
\newcommand{\Q}{{\mathbb Q}}

\newcommand{\cR}{{\mathcal R}}

\newcommand{\OO}{{\mathrm O}}

\newcommand{\Ss}{{\mathbb S}}

\newcommand{\e}{\varepsilon}

\newcommand{\yy}{{ \widetilde y}}
\newcommand{\zz}{{ \widetilde z}}
\newcommand{\Ex}{{\mathrm{exp}}}

\newtheorem{theorem}{Theorem}[section]

\newtheorem{cor}[theorem]{Corollary}

\newtheorem{definition}{Definition}[section]

\newtheorem{lem}[theorem]{Lemma}

\newtheorem{prop}[theorem]{Proposition}

\newtheorem{remark}[theorem]{Remark}

\numberwithin{equation}{section}

\title[Slow convergence phenomenon]{Slow convergence in periodic homogenization problems for divergence type elliptic operators}

\author{Hayk Aleksanyan}
\address{School of Mathematics, The University of Edinburgh, JCMB The King's Buildings, Peter Guthrie Tait Road, Edinburgh EH9 3FD}

\curraddr{Department of Mathematics, KTH Royal Institute of Technology, SE-100 44  Stockholm,
Sweden}

\email{hayk.aleksanyan@gmail.com}

\begin{document}

  \begin{abstract}

    We introduce a new constructive method for establishing lower bounds on
    convergence rates of periodic homogenization problems associated with divergence type elliptic operators.
    The construction is applied in two settings. 
    First, we show that solutions to boundary layer problems 
    for divergence type elliptic equations set in halfspaces and with infinitely smooth data, may converge to
    their corresponding boundary layer tails as slow as one wish depending on the position of
    the hyperplane. Second, we construct a Dirichlet problem for divergence type elliptic operators set in a bounded domain,
	and with all data being $C^\infty$-smooth,
	for which the boundary value homogenization holds with arbitrarily slow speed.

\end{abstract}

\maketitle

\section{Introduction}
The focus of the paper, as the title suggests, is on quantitative theory
of periodic homogenization of divergence type elliptic operators.
Lately, there has been much interest around effective estimates on convergence rates
for homogenization problems associated with linear elliptic operators in divergence form, see for example
\cite{GM-JEMS}-\cite{GM-Acta}, \cite{ASS1}-\cite{ASS3},
\cite{Kenig-Lin-Shen-L2}, \cite{Kenig-Lin-Shen-CPAM}, \cite{Ar-Shen}.
A particular common feature of these papers, is that they all establish 
upper bounds on the speed of convergence of homogenization, in other words
these papers, among other results, measure how fast the homogenization holds. 
However, results showing limitations of the speed of the process,
i.e. estimating to which extent homogenization may decelerate,
given that the homogenization takes place of course,
seem to be extremely scarce in the literature. A few instances
of this type of set-up around the divergence setting
which we are aware of are the following. It is shown in \cite{ASS2}-\cite{ASS3},
that the Dirichlet boundary value homogenization in $L^p$ can not be faster than a certain algebraic rate
depending on $1\leq p<\infty$ and on the geometry of the domain
(see \cite[Theorem~1.6]{ASS2} and \cite[Theorem 1.3]{ASS3}).
The next one studied in \cite{Prange}
is related to boundary layer problems set in halfspaces, and shows that 
depending on the position of the halfspace, convergence of the solution to its boundary layer tail
can be slower than any algebraic rate (see \cite[Theorem 1.3]{Prange}).
Finally, in \cite{BBMM} it is proved an existence of one-dimensional
examples in almost periodic homogenization with fixed boundary and source terms
and oscillating coefficients,
where homogenization of solutions in $L^2$ is not faster than a polynomial rate.

Here we will be interested in developing tools that will address how slow the convergence can actually be.
We start with the discussion of the first problem considered in this paper, and 
will introduce part of the key ideas in that setting.

\vspace{0.1cm}

For a scalar $a\in \R$ and a unit vector $n\in \R^d$
consider the following Dirichlet problem
\begin{equation}\label{bdry-layer-system-recall}
\begin{cases} -\nabla \cdot A(y) \nabla v(y) =0 ,&\text{  $y \in \Omega_{n,a}:=\{y\in \R^d: \ y\cdot n >a \}$}, \\
 v(y)=v_0(y) ,&\text{  $ y \in \partial \Omega_{n,a}  $.}    \end{cases}
\end{equation}
The main assumptions concerning \eqref{bdry-layer-system-recall} which will be in force throughout are:
\begin{itemize}[label=\textbullet] 
 \bulitem{(Periodicity)} The coefficients $A$ and boundary data $v_0$ are $\Z^d$-periodic, that is
 for any $y\in \R^d$ and any $h\in \Z^d$
 $$
 A(y+h)=A(y) \qquad \text{ and } \qquad v_0(y+h) = v_0(y),
 $$
 \bulitem{(Ellipticity)} There exists a constant $\lambda>0$ such that for any $x \in \R^d$
 and any $y\in \R^d$ one has
 $$
 \lambda |x|^2 \leq x^T A(y) x \leq \lambda^{-1} |x|^2.
 $$
 \bulitem{(Regularity)}
 All elements of $A$ and boundary data $v_0$ are infinitely smooth.
\end{itemize}

We will refer to (\ref{bdry-layer-system-recall}) as \emph{boundary layer} problem.
These type of problems emerge in the theory of periodic homogenization of Dirichlet problem
for divergence type elliptic operators with periodically oscillating coefficients and boundary data.
Understanding the well-posedness of problems of the form (\ref{bdry-layer-system-recall}) in a suitable class of solutions,
and more importantly the asymptotics of solutions far away from the boundaries
are one of the key steps toward obtaining quantitative results for homogenization of the mentioned class of Dirichlet problems.
For detailed discussions concerning the role of (\ref{bdry-layer-system-recall})
we refer the reader to \cite{GM-JEMS}, \cite{GM-Acta} and \cite{Prange}.
Below we will briefly review some known results concerning
boundary layer problems.
Interestingly, the asymptotic analysis of (\ref{bdry-layer-system-recall}) depends on certain
number-theoretic properties of the normal vector $n$.

\vspace{0.2cm}

\textbf{Rational directions.} We say that $n\in \R^d$ is a \emph{rational} vector,
and write $n\in \R \Q^d$, if $n$ is a scalar multiple
of a vector with all components being rational numbers. 
One can easily see that if $n\in \R^d$ has length one, then $n$ is rational iff $n=\xi/|\xi| $ for some non-zero $\xi\in \Z^d$.
In this case it is well-known (see e.g. \cite[Lemma 4.4]{AA99})
that there exists a smooth variational solution $v$ to (\ref{bdry-layer-system-recall}), which is unique given some decay conditions on the gradient,
and such that there is a constant $v^{a,\infty}$ for which $v(y) \to v^{a,\infty}$ exponentially fast as $y\cdot n \to \infty$,
where the convergence is uniform with respect to tangential directions.
Although having these nice convergence properties, the drawback of rational directions is that
the constant $v^{a,\infty}$ may actually depend on $a$, i.e. translating the hyperplane in the direction of $n$ may lead to different limits
at infinity.

\vspace{0.2cm}
\textbf{Diophantine directions.}
Following \cite{GM-Acta} for a unit vector $n\in \Ss^{d-1}$ set
$P_{n^\bot} $ to be the operator of orthogonal projection on the hyperplane orthogonal to $n$.
Fix $l>0$ so that $(d-1)l>1$ and for $\kappa>0$ define
\begin{equation}\label{class-diophantine}
\mathcal{A}_\kappa=\big\{ n\in \Ss^{d-1}: | P_{n^\bot}(\xi) | \geq \kappa |\xi|^{-l} \text{ for all } \xi \in \Z^d \setminus \{0\}   \big\}.
\end{equation}
A vector $n\in \Ss^{d-1}$ is called \emph{Diophantine} if it is from $\mathcal{A}_\kappa$ for some $\kappa>0$.
Clearly elements of $\mathcal{A}_\kappa$ are non-rational directions.
Also, it is not hard to verify (see \cite[Section 2]{GM-Acta}) that $\sigma(\mathbb{S}^{d-1} \setminus \mathcal{A}_\kappa)\leq C\kappa^{d-1}$,
where $\sigma$ is the surface measure of the unit sphere, and $C$ is a constant depending on $l$. Thus, the last inequality shows that almost all directions are Diophantine.

Behaviour of (\ref{bdry-layer-system-recall}) in the case when $n$ is Diophantine
has been studied only recently in \cite{GM-JEMS}, where it was proved (Propositions 4 and 5 of \cite{GM-JEMS})
that there exists a smooth variational solution $v$ to (\ref{bdry-layer-system-recall}) which
is unique, given some growth conditions, and such that for some constant $v^\infty$ one has
$v(y) \to v^\infty$ as $y \cdot n \to \infty$. Here convergence is locally uniform with respect to tangential
directions, and is faster than any polynomial rate in $y\cdot n$. Moreover,
the effective constant $v^\infty$ (the \emph{boundary layer tail})
depends on direction $n$ only, and is independent of $a$ in contrast to the rational case.

\vspace{0.2cm}
\textbf{Non-rational directions in general.} Here we consider directions which are \emph{irrational}, i.e.
from the complement of $\R \Q^d$. Observe that not all irrational directions are Diophantine,
therefore the previous two cases do not cover $\mathbb{S}^{d-1}$, the set of all possible directions. 
In a recent work \cite{Prange}, it was proved that \eqref{bdry-layer-system-recall} has a unique smooth variational solution 
satisfying certain growth conditions, for which one has convergence toward its boundary layer tail
far away from the boundaries (see Theorem \ref{Thm-Prange} for the precise statement).
However, the result of \cite{Prange} does not provide any estimates on the rate of convergence
given this generality on the normals.
It does however show that for irrational directions which are non-Diophantine (meaning they fail
to satisfy \eqref{class-diophantine} for any choice of parameters $\kappa $ and $l$ involved in the definition), one may have convergence
slower than any power rate in $y\cdot n$. More precisely, 
for a smooth and $\Z^2$-periodic function $v_0: \mathbb{T}^2 \to \R$ consider the following boundary value problem
\begin{equation}\label{u-on-Omega-n}
\Delta v =0 \ \text{in } \Omega_n \qquad \text{ and } \qquad v= v_0 \ \text{on } \partial \Omega_n,
\end{equation}
where $\Omega_n = \{x\in \R^2: \ x\cdot n>0 \}$. 
Clearly, this problem is of type \eqref{bdry-layer-system-recall} with the matrix of coefficients equal
to $2\times 2$-identity matrix. Then \cite[Theorem 1.3]{Prange} shows that if $n\notin \R \Q^2$ is
arbitrary non-Diophantine direction, then for any $p>0$ and any $R>0$ there exists a smooth function $v_0: \mathbb{T}^2 \to \R$ and 
a sequence $\lambda_k \nearrow \infty$ such that if $v$ solves (\ref{u-on-Omega-n}) with boundary data $v_0$, then
for each $k=1,2,...$ and all $y' \in \partial \Omega_{n} \cap B(0, R)$ one has
\begin{equation}\label{slow-conv-Prange}
| v(y' + \lambda_k n) - v^\infty | \geq \lambda_k^{-p},
\end{equation}
where the constant $v^\infty$ is the corresponding boundary layer tail. Let us note that 
the left-hand side of (\ref{slow-conv-Prange}) converges to zero as $k\to \infty$, since as we have just said, for irrational directions
solutions converge to their boundary layer tails.
The proof of (\ref{slow-conv-Prange}) constructs $v_0$ with Fourier spectrum supported in a subset of
$\Z^2$ on which the normal $n$ fails to satisfy the Diophantine condition.
Then choosing coefficients having an appropriate decay, combined with the special structure
of the spectrum of $v_0$, immediately leads to the conclusion.
We stress that this slow convergence result of \cite{Prange} works for any irrational non-Diophantine direction,
however, it leaves out the question whether one can go beyond algebraic rates of convergence,
and perhaps more intriguing, the case of Laplace operator does not give an insight into the case of variable coefficient operators,
since in the Laplacian setting one has an explicit form of solutions which essentially determines the analysis.

\vspace{0.3cm}

\noindent Throughout the paper we use the following notation and conventions.
\vspace{0.3cm}

\begin{tabbing}
\hspace{1.7cm}\=\kill
 \vspace{0.1cm}
 $\mathbb{T}^d$ \> unit torus of $\R^d$, i.e. the quotient space $\R^d/ \Z^d$, where $\Z^d$ is the integer lattice, \\
 \vspace{0.1cm}
 $\mathbb{S}^{d-1}$ \> unit sphere of $\R^d$, \\
 \vspace{0.1cm}
 $\OO(d)$ \> group of $d\times d$ orthogonal matrices, \\
 \vspace{0.1cm}
 $M^T$ \> transpose of a matrix $M$ \\
 \vspace{0.1cm}
 $x\cdot y$ \> inner product of $x,y\in \R^d$ \\
  \vspace{0.1cm}
  $|x|$ \> Euclidean length of $x\in \R^d$ \\
  \vspace{0.1cm}
 $\Omega_{n,a}$  \> halfspace  $\{ x\in \R^d: \ x\cdot n>a \}$, where $n\in \mathbb{S}^{d-1}$ and $a\in \R$, \\
 \vspace{0.1cm}
 $\Omega_n$ \> halfspace $\Omega_{n,0}$, \\
 \vspace{0.1cm}
 $B_r(x)$ \> or $B(x,r)$ both stand for an open ball with center at $x\in \R^d$ and radius $r>0$,  \\ 
 \vspace{0.1cm}
  $\Subset$ \> compact inclusion for sets.
\end{tabbing} 

\vspace{0.3cm}
Positive generic constants are denoted by various letters $C, C_1,c,...$,
and if not specified, may vary from formula to formula.
For two quantities $x,y$ we write $x\lesssim y$ for the inequality $x\leq C y$
with an absolute constant $C$, and $x\asymp y$ for double inequality $C_1  x \leq y \leq C_2 x$
with absolute constants $C_1,C_2$. 

Throughout the text the word ``smooth"
unless otherwise specified, means differentiable of class $C^\infty$.
The term ``modulus of continuity'' is everywhere understood in accordance with Definition \ref{def-mod-contin}.
The phrase ``boundary layer tail" refers to the constant determined by Theorem \ref{Thm-Prange}.
\emph{Domain} is an open and connected subset of Euclidean space.
We also adopt the summation convention of repeated indices (Einstein summation convention).

\subsection{Main results}

The first class of problems we will study in this article,
is motivated by the results discussed above, and the importance of boundary layer problems in periodic homogenization
of Dirichlet problem. Most notably, we will show that in the case
of irrational non-Diophantine normals convergence of solutions to (\ref{bdry-layer-system-recall})
towards their boundary layer tails can be arbitrarily slow.
Next, in the second part of the paper, we will apply our methods developed for the analysis of 
(\ref{bdry-layer-system-recall}) combined with some ideas from our papers \cite{ASS1}-\cite{ASS3}
written in collaboration with H. Shahgholian, and P. Sj\"{o}lin, to construct a Dirichlet problem
for elliptic operators in divergence form set in a bounded domain, where boundary value homogenization
holds with a speed slower than any given rate in advance.

We now proceed to formulations of the main results.
In order to measure a speed of convergence we consider the following class of functions.

\begin{definition}\label{def-mod-contin}
We say that a function $\omega$ is a \verb|modulus of continuity| if it has the following properties:
\begin{itemize}[label=\textbullet] 
\bulitem{} $\omega :[0,\infty) \to (0,\infty)$ is continuous,
\vspace{0.1cm}
\bulitem{} $\omega$ is one-to-one, decreasing, and $\lim\limits_{t \to \infty} \omega(t) =0$.
\end{itemize}
\end{definition}

We will at places abuse the notation and instead of $[0,\infty)$ may take $[c_0, \infty)$ for some $c_0>0$
as a domain of definition for modulus of continuity.

For our first result we will impose a structural restriction on coefficients $A$.
Namely we assume that
\begin{equation}\label{cond-on-A}
 \text{ there exists } \ \  1\leq \gamma \leq d \ \ \text{ such that } \ \ \partial_{y_\alpha} A^{\gamma \alpha } \equiv 0.
\end{equation}
In other words we require one of the columns of $A$ to be divergence free
as a vector field. This assumption is technical and is due to our proof of Theorem \ref{Thm-slow-variable}.
It is used to control the contribution of boundary layer correctors in the asymptotics of boundary layer tails
(see, in particular, inequality (\ref{Z15})). 

The following is our main result concerning the slow convergence phenomenon of boundary layers.

\begin{theorem}\label{Thm-slow-variable}
Let $\omega$ be any modulus of continuity and let $R>0$ be fixed. Then, there exists a unit vector $n\notin \R\Q^d$,
a smooth function $v_0:\mathbb{T}^d \to \R$,
and a sequence of positive numbers $\{\lambda_k\}_{k=1}^\infty$ growing to infinity,
such that if $v$ solves (\ref{bdry-layer-system-recall}) under condition (\ref{cond-on-A})
on the operator, and with $n$ and $v_0$ as specified here,
then for any $k=1,2,...$ and any $y'\in \partial \Omega_{n,0} \cap B(0,R)$ one has
\begin{equation}\label{z1}
| v(y' + \lambda_k n ) - v^\infty | \geq \omega(\lambda_k), 
\end{equation}
where $v^\infty$ is the corresponding boundary layer tail.
\end{theorem}

\begin{remark}
Observe that $v^\infty$ being the boundary layer tail, implies that the left-hand side of (\ref{z1})
decays as $k\to \infty$, and hence the lower bound is non-trivial.
Next, notice that we have fixed the halfspace on the direction $n\in \mathbb{S}^{d-1}$ by setting $a=0$ in Theorem \ref{Thm-slow-variable}.
This does not lessen the generality, since the case of arbitrary $a\in \R$ can be recovered by a change of variables.
However, we do not know if in general one can take the sequence $\{\lambda_k\}$ and boundary data $v_0$ independently
of $a$. 
\end{remark}
 
Finally, let us note that the result of 
Theorem \ref{Thm-slow-variable} shows that there is \emph{no} lower bound for the speed of convergence on the set of irrational directions,
in other words convergence can be in fact arbitrarily slow.
This case is in strong contrast with the case of Diophantine normals, where one has convergence faster than any power rate.

\vspace{0.2cm}

Our next concern is the question of Dirichlet boundary value homogenization for 
divergence type elliptic operators in bounded domains. 
Assume we are given a coefficient matrix 
$A=(A^{\alpha \beta}(x))_{\alpha , \beta=1}^d: X \to \R^{d\times d}$,
defined on some domain $X\subset \R^d$ ($d\geq 2$) and having these properties:
\begin{itemize}
 \item[(A1)] for each $1\leq \alpha , \beta \leq d$ we have $A^{\alpha \beta} \in C^\infty(X)$,
\vspace{0.2cm}
 \item[(A2)] there exist constants $0<\lambda\leq \Lambda<\infty$ such that
$$
\lambda |\xi|^2 \leq A^{\alpha \beta} (x) \xi_\alpha \xi_\beta \leq \Lambda |\xi|^2,
 \qquad \forall x\in X, \ \forall  \xi \in \R^d.
$$
\end{itemize}
For a function $g\in C^\infty( \mathbb{T}^d)$,
and a bounded subdomain $D\subset X$ with $C^\infty$ boundary
consider the following problem
\begin{equation}\label{Dirichlet-bdd-domains}
 -\nabla \cdot A(x) \nabla u_\e(x) = 0 \text{ in } D \qquad \text{ and } \qquad u_\e(x) = g (x/ \e) \text{ on } \partial D,
\end{equation} 
where $\e>0$ is a small parameter. Along with \eqref{Dirichlet-bdd-domains} consider the corresponding homogenized problem
which reads
\begin{equation}\label{Dirichlet-bdd-domains-homo}
 -\nabla \cdot A(x) \nabla u_0(x) = 0 \text{ in } D \qquad \text{ and } \qquad u_0(x) = \int_{\mathbb{T}^d} g(y) dy \text{ on } \partial D.
\end{equation} 

Let us emphasize that we do \emph{not} impose any structural restrictions on $A$,
nor we assume that $A$ is necessarily periodic. This is in view of the fact
that there is no interior homogenization taking place in (\ref{Dirichlet-bdd-domains}).

\begin{theorem}\label{Thm-slow-Dirichlet}
Let $A$ be as above defined on $X$ and satisfy (A1)-(A2), and let
$\omega$ be any modulus of continuity.
Then, there exist bounded, non-empty convex domains $D\subset X$ and $D'\Subset D$ with $C^\infty$ boundaries,
and a real-valued function $g \in C^\infty(\mathbb{T}^d)$ such that if
$u_\e$ is the solution to \eqref{Dirichlet-bdd-domains} for $\e>0$, and $u_0$ to that of \eqref{Dirichlet-bdd-domains-homo},
then for some sequence of positive numbers $\{\e_k\}_{k=1}^\infty$ strictly decreasing to 0, 
one has the following:
\begin{itemize}
\item[{\normalfont a)}] $  |u_{\e_k}(x) - u_0(x)| \geq  \omega(1/ \e_k ), \qquad \forall x\in D', \ k=1,2,...,$
\vspace{0.1cm}
\item[{\normalfont b)}] $|u_\e(x) - u_0(x)|\to 0 \ \   \text{ as } \e \to 0 , \ \  \forall x\in D$.
\end{itemize}
\end{theorem}

\vspace{0.2cm}

This result should be compared with \cite{ASS1}-\cite{ASS3}, where under certain geometric conditions 
on boundary of the domain (such as \emph{strict} convexity, or flat pieces having Diophantine normals
as considered in \cite{ASS3})
it is proved that periodic homogenization of boundary value problems for elliptic
operators in divergence form holds pointwise, as well as in $L^p$ norm where $1\leq p <\infty$, with an algebraic rate in $\e$. However here, we see
that again due to the geometry of the domain, convergence can slow down arbitrarily.
Thus relying merely on the smoothness of the data
involved in the problem, one can not get a meaningful quantitative theory for homogenization problems with divergence structure
as above.


\section{Preliminaries on solutions to boundary layer problems}\label{sec-sol-prelim}

The aim of this section is to give a precise meaning to a solution of problem (\ref{bdry-layer-system-recall}).
The well-posedness of (\ref{bdry-layer-system-recall}) in non-rational setting
was first studied by G\'{e}rard-Varet and Masmoudi  \cite{GM-JEMS}, and later
by the same authors in \cite{GM-Acta}, and by Prange in \cite{Prange}, all in connection with homogenization of Dirichlet problem.
Here for the exposition we will follow mainly \cite{GM-Acta} and \cite{Prange}.
Let us also note that the results presented in this section hold for strictly elliptic systems,
however we will only use them for scalar equations, and thus formulate the results in the setting
 of a single equation only.

Keeping the notation of problem \eqref{bdry-layer-system-recall}, fix some matrix $M\in \OO(d)$
such that $M e_d= n$. Then in (\ref{bdry-layer-system-recall})
make a change of variables by $y=Mz$ and
transform the problem to
\begin{equation}\label{Z9}
\begin{cases} -\nabla_z \cdot B(Mz) \nabla_z \textbf{v}(z) =0 ,&\text{  $z_d>a $}, \\
  \textbf{v} (z)= v_0( Mz ) ,&\text{  $ z_d= a  $,}    \end{cases}
\end{equation}
where $\textbf{v}(z) = v(Mz)  $ and the new matrix $B$ is defined by
$B = M^T A  M$. 
From the definition of $M$ we have $M=[N| n]$, where $N$ is a matrix of size $d\times (d-1)$. Then
the solution to (\ref{Z9}) is being sought of the form
$$
\textbf{v}(z) = V(Nz', z_d),  \ \text{ where } \ V(\cdot, t) \text{ is } \Z^d \text{-periodic for any } t\geq a,
$$
and $z=(z', z_d)\in \R^{d-1}\times \R$. This leads to the following problem
\begin{equation}\label{bl-system-d+1}
\begin{cases}
  - \left(
    \begin{array}{c}
      N^T \nabla_\theta \\
      \partial_t \\
    \end{array}
  \right) \cdot B(\theta + tn) \left(
                                         \begin{array}{c}
                                           N^T \nabla_\theta \\
                                           \partial_t \\
                                         \end{array}
                                       \right) V(\theta, t) =0
       ,&\text{  $ t > a $}, \\
 V(\theta, t) = v_0 ( \theta + a n ), &\text{  $ t = a $}.    \end{cases}
\end{equation}
The authors of \cite{GM-Acta} then show that (\ref{bl-system-d+1}) has a smooth solution $V$
in the infinite cylinder $\mathbb{T}^d \times [a,\infty)$ satisfying
certain energy estimates. Moreover, if $V$ solves (\ref{bl-system-d+1}), then one can easily see that $\textbf{v}(z) = V(Nz', z_d)$
gives a solution to \eqref{Z9}. The proof of this well-posedness result is not hard,
but what is rather involved is the analysis of asymptotics of $V(\cdot , t)$ as $t \to \infty$.
A proper understanding of this problem for $V$ gives the behaviour of solutions to (\ref{bdry-layer-system-recall})
far away from the boundary of the corresponding halfspace.
In this regard, it was proved in \cite{GM-JEMS}
that if $n$ is Diophantine in a sense of (\ref{class-diophantine}), then there exists a constant $v^\infty$ depending on $n$ and independent of $a$, 
such that $|v(y)-v^\infty|\leq C_\alpha (y\cdot n)^{-\alpha} $, for any $\alpha>0$ as $y\cdot n \to \infty$, and convergence is locally
uniform with respect to tangential variables.

Shortly after \cite{GM-JEMS} and \cite{GM-Acta}, a refined analysis of well-posedness of problems of type (\ref{bdry-layer-system-recall})
was given by Prange \cite{Prange}. In particular he established the following result. 
\begin{theorem}\label{Thm-Prange} {\normalfont{(Prange \cite[Theorem 1.2]{Prange})}}
Assume $n \notin \R \Q^d$. Then
\begin{itemize}
 \item[{\normalfont{1.}}] there exists a unique solution $v\in C^\infty(\overline{\Omega_{n,a} }) \cap L^\infty (\Omega_{n,a})$ of (\ref{bdry-layer-system-recall}) such that
$$
\qquad || \nabla v ||_{L^\infty(\{y \cdot n >t\} ) } \to 0, \text{ as } t \to \infty,
$$

$$
\int_a^\infty || (n \cdot \nabla) v ||_{L^\infty(\{ y\cdot n -t=0 \})  }^2 dt <\infty,
$$

\item[{\normalfont{2.}}] and a boundary layer tail $v^\infty \in \R$ independent of $a$ such that 
$$
v(y) \to v^\infty, \text{ as } y\cdot n \to \infty, \text{ where } y\in \Omega_{n,a},
$$
and convergence is locally uniform with respect to tangential directions.
\end{itemize}

\end{theorem}

Let us fix here that by \emph{boundary layer tail} we always mean the constant to which solutions of
problem (\ref{bdry-layer-system-recall}) converge away from the boundary of the corresponding halfspace.


\section{Arbitrarily slow convergence for Laplacian}\label{sec-Laplace}

The objective of the present section is to prove Theorem \ref{Thm-slow-variable} for Laplacian.
The reason for separating the case of Laplace operator is twofold. First, we will
introduce part of the key ideas that will be used in the general case of variable coefficient operators.
Second, the setting of Laplacian is essentially self-contained, and is more transparent
in comparison to the general case which is based on a different approach.
We prove the following result.

\begin{theorem}\label{Thm-slow-conv-Laplace}
Theorem \ref{Thm-slow-variable} holds when the operator in (\ref{bdry-layer-system-recall}) is the Laplacian.
\end{theorem}


The proof of Theorem \ref{Thm-slow-conv-Laplace} is based on a series of observations and preliminary statements.
We will use the connection of problem (\ref{u-on-Omega-n}) with the corresponding problem \eqref{bl-system-d+1}
set on a cylinder $\mathbb{T}^d \times [0,\infty)$ (cf. \cite[Theorem 7.1]{Prange}). For that fix a matrix $M\in \OO (d)$ such
that $M e_d = n$, clearly $M$ is of the form
$M=[N | n]$ where $N$ is a matrix of size $d\times (d-1)$. 
Then writing (\ref{bl-system-d+1}) when the original operator in (\ref{bdry-layer-system-recall})
is Laplacian and $a=0$ we get
\begin{equation}\label{V-of-u}
\begin{cases} \begin{vmatrix}
 N^T \nabla_\theta  \\
\partial_t 
\end{vmatrix}^2 V(\theta , t) =0 ,&\text{  $t>0, \ \theta \in \mathbb{T}^d $}, \\
 V(\theta , 0) =v_0(\theta) ,&\text{  $  \theta \in \mathbb{T}^d  $,}    \end{cases}
\end{equation}
where as before $V(\cdot, t)$ is $\Z^d$-periodic for all $t\geq 0$, and the action of the operator on $V$
is understood as $(N^T \nabla_\theta, \partial_t )\cdot (N^T \nabla_\theta, \partial_t) V$. As we have discussed
in Section \ref{sec-sol-prelim}, the unique solution $v$ of \eqref{u-on-Omega-n} (in a sense of Theorem \ref{Thm-Prange}) is given by
\begin{equation}\label{v-in-terms-of-V}
 v(y) = v(Mz) = V(Nz', z_d), \text{ where } y=Mz \text{ with } y\in \Omega_n  \text{ and } z\in \R^d_+.
\end{equation}

\noindent In this setting the solution $V$ of \eqref{V-of-u} can be computed explicitly, namely
we have
\begin{equation}\label{V-series}
V(\theta, t) = \sum\limits_{\xi \in \Z^d} c_\xi (v_0) e^{  -2\pi |N^T \xi| t } e^{ 2\pi i \xi \cdot \theta },
\end{equation}
where $c_\xi(v_0)$ is the $\xi$-th Fourier coefficient of $v_0$. In view of (\ref{V-series}) it is clear the
boundary layer tail equals $c_0(v_0)$. 
We will first establish a slow convergence result for $V$ using
which we will prove Theorem \ref{Thm-slow-conv-Laplace}.

Observe that by (\ref{V-series}) and Parseval's identity, for any $t\geq 0$ we have
\begin{multline}\label{Parseval}
|| V(\theta, t) - c_0(v_0) ||_{L^\infty( \mathbb{T}^d ) }^2  \geq   || V(\theta, t) - c_0(v_0) ||_{L^2( \mathbb{T}^d ) }^2 = \\
\sum\limits_{\xi \in \Z^d \setminus \{0\}} |c_\xi(v_0)|^2 e^{-4\pi |N^T \xi| t} := \mathcal{S}(t; v_0).
\end{multline}

\begin{prop}\label{Thm-slow-conv}
For any modulus of continuity $\omega$
there exists a unit vector $n\notin \R \Q^d$, a smooth function $v_0 : \mathbb{T}^d \to \R$, and a sequence
of positive numbers $t_k \nearrow \infty$, $k=1,2,...$ such that
$$
\mathcal{S}(t_k; v_0 ) \geq \omega( t_k), \qquad k=1,2,...,
$$
where $\mathcal{S}$ is defined by (\ref{Parseval}).
\end{prop}

As we can observe from (\ref{Parseval}) convergence properties of $\mathcal{S}$
depend on the quantity $|N^T \xi|$ which is the subject of the next result.

\begin{lem}\label{Lem-bad-normals} 
Given any modulus of continuity $\omega$,
there exists a unit vector $n\notin \R \Q^d$, and an infinite set $\Lambda \subset \Z^d \setminus \{ 0\}$
such that for any matrix $M=[N |n] \in \OO(d)$ one has
$$
| N^T \xi | \leq \omega(|\xi |), \qquad \forall \xi \in \Lambda.
$$
\end{lem}

\vspace{0.2cm}
As usual here as well $N$ is a $d\times (d-1)$-matrix formed from the first $(d-1)$ columns of $M$.
Obviously $M e_d=n$. This lemma is one of the key statements used in the general case as well.
\vspace{0.1cm}

\begin{proof}[Proof of Lemma \ref{Lem-bad-normals}]
Set $\xi^{(1)}: = e_1 = (1,0,...,0)\in \R^d $ and let $\Gamma_1 \subset \Ss^{d-1}$
be an open neighbourhood of $\xi^{(1)}$ centred at $\xi^{(1)}$ and with diameter less than
$\omega^2( | \xi^{(1)} | ) /  ( 10 |\xi^{(1)}|^2  )  $. Due to the density of rational
directions\footnote{For any non-zero $\nu \in \Q^d$ the intersection of the ray starting at $0\in \R^d$ and passing through $\nu$, with the sphere $\mathbb{S}^{d-1}$ is a rational vector of unit length.
Hence the density of rational directions.} there exists a non-zero
$\xi^{(2)} \in \Z^d$ such that $|\xi^{(2)}| \geq 2$ and 
$$
0< \left| \frac{ \xi^{(2)} }{ |\xi^{(2)}|  } - \frac{ \xi^{(1)} }{ |\xi^{(1)}|  } \right| 
\leq \mathrm{diam}(\Gamma_1) \leq \frac{\omega^2( | \xi^{(1)} | )}{ 10 |\xi^{(1)}|^2  }.
$$
Using the same idea, we then inductively construct a sequence $\{ \xi^{(k)} \} \subset \Z^d \setminus \{0\}$
satisfying $k \leq | \xi^{(k)} | < |\xi^{(k+1)}| $ and such that for unit vectors $r_k = \frac{ \xi^{(k)} }{ |\xi^{(k)}|  }$
we get
\begin{equation}\label{xi-k}
  0< \left| r_{k+1} - r_k  \right| \leq \frac{\omega^2( | \xi^{(k)} | )}{ 10^k |\xi^{(k)}|^2  }
\end{equation}
for each $k=1,2,...$ . It is clear by (\ref{xi-k}) that for $k$ large enough one has 
$|r_{k+1} -r_k| <10^{-k}$, therefore the sequence $\{ r_k \}$ is Cauchy, hence it is convergent.
By $n$ we denote its limit, which is obviously a unit vector.
We claim that $n\notin \R\Q^d$, which is
due to fast convergence of the sequence\footnote{This is in analogy with a standard fact in Diophantine approximation theory, that the sum of a very fast converging series of rationals is irrational.} $\{r_k\}$.
Indeed, assume for contradiction that $n\in \R \Q^d$. As $|n|=1$, using the rationality
assumption it is easy to see that there exists a non-zero $\xi_0 \in \Z^d$ such that
$n=\xi_0 / |\xi_0|$. By monotonicity of $|\xi^{ (k) }|$ and $\omega$, along with (\ref{xi-k}) we get
\begin{multline*}
 | (n\cdot e_1)^2  - (r_k \cdot e_1)^2  | \leq 2 | n\cdot e_1  - r_k \cdot e_1  | \leq 2| n   - r_k  | \leq 
 2\sum\limits_{j=k}^\infty  | r_{j+1 } - r_j | \leq \\ 
 2 \sum\limits_{j=k}^\infty  \frac{\omega^2( |\xi^{ (j) }|  )}{10^j |\xi^{ (j) }|^2  } \leq
 \frac{20}{9} \frac{1}{10^k} \frac{\omega^2( |\xi^{ (k) }|  )}{  |\xi^{ (k) }|^2  }, \ \ k=1,2,... \ .
\end{multline*}
By rationality of $n$ we have $n\cdot e_1 = \frac{p_0}{ |\xi_0| }$ with $p_0  \in \Z $, and for $r_k$ we have
$ r_k \cdot e_1 = \frac{p_k}{ |\xi^{ (k) }| } $ with some $p_k\in \Z$. Hence, from the last inequality we get
$$
\biggl| p_0^2  |\xi^{ (k) }|^2     - p_k^2 |\xi_0 |^2  \biggr| \leq \frac{20}{9} \frac{1}{10^k} \omega^2( |\xi^{ (k) }|  ) |\xi_0 |^2.
$$
Since the left-hand side is an integer and is less than 1 by absolute value
for $k$ large enough, it must be 0. From here we conclude that
the sequence $|r_k \cdot e_1 |$ is eventually constant. By our notation this implies
\begin{equation}\label{p-k-const}
\frac{p_k}{ | \xi^{(k)} | } = \pm \frac{p_{k+1}}{ | \xi^{(k+1)} | } , \ k \geq k_0 ,
\end{equation}
where $p_k$ is an integer, representing the first coordinate of $\xi^{(k)}$
and $k_0$ is a large integer. Now, if we have equality 
in the last expression with minus sign, we get
$$
\frac{ 1 }{ 10^k | \xi^{(k)} |^2 } \geq | r_{k+1} - r_k | \geq | r_{k+1} \cdot e_1 - r_k \cdot e_1 | = \frac{2 |p_k|}{ | \xi^{(k)} | },
$$
which implies that $p_k=0$, and hence $p_{k+1}$ is 0 as well by \eqref{p-k-const}. We thus see that
(\ref{p-k-const}) in either case of the signs, forces equality within the first components of
$r_k$ and $r_{k+1}$. But in this case the same argument with $e_1$ replaced by the remaining vectors
of the standard basis of $\R^d$, would lead to equality 
for all corresponding components of $r_k$ and $r_{k+1}$.
The latter contradicts the fact that $r_k \neq r_{k+1}$ which we have from (\ref{xi-k}).
Hence the proof that $n$ is not a rational direction is complete.

We now set $\Lambda = \{\xi^{(k)}: \ k=1,2,... \}$ and proceed to the proof of the claimed estimate of the lemma.
By orthogonality of $M$ for any $\xi \in \Lambda $ we have
$$
|\xi|^2 = |M^T \xi|^2  = | N^T \xi |^2 + |n\cdot \xi|^2,
$$
which, combined with Cauchy-Schwarz, implies
\begin{equation}\label{N-T-square}
| N^T \xi |^2 = | \xi |^2 - |n\cdot \xi|^2  =  \big(| \xi |  + |n\cdot \xi| \big) \big( | \xi |  - |n\cdot \xi| \big) \leq 
2 | \xi |^2 \left(  1-  \frac{|n\cdot \xi|}{|\xi|}  \right).
\end{equation}
Now choose $k\in \N$ such that $\xi=\xi^{(k)}$. We get
$$
\xi \cdot n =  \xi^{(k)}  \cdot n =  \xi^{(k)} \cdot \left[  \frac{\xi^{(k)}}{ |\xi^{(k)}|  }  +  \sum\limits_{j=k}^\infty (r_{j+1} - r_j )  \right] =
|\xi^{(k)}| + \sum\limits_{j=k}^\infty \xi^{(k)} \cdot (r_{j+1} - r_j )  .
$$
Hence by (\ref{xi-k}) we have
$$
\big| \xi^{(k)}  \cdot n  -  |\xi^{(k)}|  \big| \leq \sum\limits_{j=k}^\infty | \xi^{(k)}| | r_{j+1}  -r_j | \leq 
\sum\limits_{j=k}^\infty | \xi^{(k)}|  \frac{\omega^2(|\xi^{(j)}|)}{ 10^j | \xi^{(j)} |^2 } \leq \frac{1}{2}\frac{\omega^2(|\xi^{(k)}|) }{|\xi^{(k)}|}.
$$
We thus get
$$
\left|  1-  \frac{|n\cdot \xi^{(k)}|}{|\xi^{(k)}|}  \right| \leq \frac{\omega^2(|\xi^{(k)}|) }{2 |\xi^{(k)}|^2}.
$$
From here, getting back to (\ref{N-T-square}) we obtain
$$
| N^T \xi |^2 \leq \omega^2(|\xi|),
$$
completing the proof of the lemma.
\end{proof}

The following remarks will be used later on.
\begin{remark}\label{rem-large-gap}
One may easily observe from the proof and the density of rational directions, that given any $\tau>1$ it is possible to construct $\Lambda$ such that
for any $\xi,\eta \in \Lambda $ if $|\xi|<|\eta|$ then $\tau |\xi|<|\eta|$. 
\end{remark}

\begin{remark}\label{rem-bad-normals-dense-measure-0}
The set of normal directions satisfying Lemma \ref{Lem-bad-normals} is dense on $\mathbb{S}^{d-1}$, however it necessarily has measure 0
if $\omega$ decreases faster than any polynomial rate. 
Density follows from the proof of the lemma,
as there one may start the construction in the neighbourhood of any rational direction on $\mathbb{S}^{d-1}$ instead of $e_1$. 
The measure zero claim is due to the fact that the set of Diophantine directions, in a sense of (\ref{class-diophantine}),
has full measure on $\mathbb{S}^{d-1}$,
and if $\omega$ decreases sufficiently fast, then any Diophantine direction fails to satisfy Lemma \ref{Lem-bad-normals}.
\end{remark}

We now give a proof of Proposition \ref{Thm-slow-conv} based on the previous lemma.

\bigskip

\begin{proof}[Proof of Proposition \ref{Thm-slow-conv}]
Define $\omega_1(t): [1,\infty) \to \R_+$ as follows
$$
\omega_1(t) = \frac{1}{ 4 \pi \omega^{-1} \left( \frac{1}{e}  t^{-2t}  \right) }, \qquad t\geq 1.
$$
Here $\omega^{-1}$ stands for the inverse function of $\omega$, which exists since $\omega$ is one-to-one.
Moreover, without loss of generality we will assume that $\omega_1$ is well-defined for $t\geq 1$,
since otherwise we will just replace the lower bound of $t$ by a sufficiently large number.
It is easy to see from the definition of $\omega_1(t)$ that it is continuous, decreases to 0 as $t \to \infty$ and 
is one-to-one. We now apply Lemma \ref{Lem-bad-normals} for this choice of $\omega_1$ as a modulus of continuity,
and let $n$ be the normal, and $\Lambda$ be the index set given by Lemma \ref{Lem-bad-normals}.
We define $v_0\in C^\infty(\mathbb{T}^d)$ as follows. First arrange elements of $\Lambda$
in increasing order of their norms, i.e. we let $\Lambda = \{ \xi^{(k)}: \ k=1,2,... \}$,
where $| \xi^{(k)}|< | \xi^{(k+1)} |$, for $k=1,2,...$ . Observe that $|\xi^{(k)}|\geq k$ for all $k\in \N$
due to the construction of Lemma \ref{Lem-bad-normals}.
For $\xi \in \Z^d$ we let $c_\xi$ be the $\xi$-th Fourier coefficient of $v_0$.
Next, if $\xi \in \Lambda$ is the $k$-th element of $\Lambda$ according to the mentioned arrangement, set $c_{ \pm \xi}(v_0)= |\xi|^{-k}$,
otherwise, if $\pm \xi \notin \Lambda$ let $c_{\xi}(v_0)=0$.
The sequence of coefficients constructed in this way defines a smooth function $v_0$, since the Fourier coefficients of $v_0$ decay
faster than any polynomial rate. Furthermore, as $c_{-\xi} = c_\xi \in \R$, $v_0$ is a real-valued function.

By (\ref{Parseval}) we get
$$
\mathcal{S}(t; v_0) = \sum\limits_{k=1}^\infty \frac{2}{ |\xi^{(k)}|^{2k}  } e^{ -4\pi |N^T \xi^{(k)}| t }.
$$
For each $k\in \N$, choose $t_k$ from the condition that 
\begin{equation}\label{g-t-k}
\omega(t_k) = \frac{1}{e} |\xi^{(k)}|^{-2k}, \qquad k=1,2,... \ .
\end{equation}
By construction we have $t_k \nearrow \infty$. To prove that 
$\mathcal{S}(t_k; v_0) \geq \omega(t_k) $ it is enough to show that $ e^{ -4\pi |N^T \xi^{(k)}| t_k }  |\xi^{(k)}|^{-2k}  \geq \omega(t_k)  $,
while for this one it is sufficient to prove
\begin{equation}\label{inq-with-omega}
e^{ -4\pi \omega_1 (|\xi^{(k)}|) t_k }  |\xi^{(k)}|^{-2k}  \geq \omega(t_k)  ,
\end{equation}
since $ |N^T \xi^{(k)}|  \leq \omega_1(|\xi^{(k)}| )$ by Lemma \ref{Lem-bad-normals}. We now use (\ref{g-t-k}) and the definition of $\omega_1$,
by which (\ref{inq-with-omega}) is equivalent to 
\begin{multline*}
1 \geq 4\pi \omega_1(|\xi^{(k)}|) t_k = 4\pi \omega_1(|\xi^{(k)}| ) \omega^{-1} \left(  \frac{1}{e}  |\xi^{(k)}|^{-2k} \right) = \\
4\pi \frac{1}{ 4 \pi \omega^{-1} \left( \frac 1e  | \xi^{(k)}|^{-2 |\xi^{(k)}| }  \right) } \omega^{-1} \left(  \frac{1}{e}  |\xi^{(k)}|^{-2k} \right) .
\end{multline*}
But the last expression is equivalent to 
$$
 \omega^{-1} \left( \frac 1e  | \xi^{(k)}|^{-2 |\xi^{(k)}| }  \right) \geq  \omega^{-1} \left(  \frac{1}{e}  |\xi^{(k)}|^{-2k} \right),
$$
which holds true, since $| \xi^{(k)}|\geq k $ and $\omega^{-1}$ is decreasing.

The proof of the proposition is complete.

\end{proof}

\begin{remark}\label{rem1}
It is clear from the proof of Proposition \ref{Thm-slow-conv} that given any $\delta>0$ in advance,
we may drop some finite number of initial terms from $\Lambda \subset \Z^d$, the Fourier spectrum of $v_0$,
ensuring that $|N^T \xi| \leq \delta $ for all $\xi \in \Lambda$. 
\end{remark}

\begin{remark}
By the same way as in Remark \ref{rem-bad-normals-dense-measure-0} we may argue that the set of normals
with the property as discussed in Proposition \ref{Thm-slow-conv}, is dense on the unit sphere,
however with measure zero if $\omega$ has a sufficiently slow decay at infinity.
It is also clear that any prescribed lower bound on the sequence $|N^T \xi|$
for non-zero $\xi\in \Z^d$ by a given function of $|\xi|$
would transform to a certain upper bound on the speed of convergence of $V$ to its tail.
\end{remark}

\vspace{0.2cm}

\begin{proof}[Proof of Theorem \ref{Thm-slow-conv-Laplace}]
Let a unit vector $n\notin \R \Q^d$, a $\Z^d$-periodic function $v_0$
and a sequence of positive numbers $\{t_k\}_{k=1}^\infty$ be obtained by applying Proposition \ref{Thm-slow-conv}
for the modulus of continuity $\omega$ given in Theorem \ref{Thm-slow-conv-Laplace}. 
Also let $V$ be defined by \eqref{V-series} for this choice of $v_0$. As we have seen in (\ref{v-in-terms-of-V}) the unique solution $v$ to
problem \eqref{u-on-Omega-n} is given by
$ v(y)  = V(Nz', z_d)$, where as usual $z\in \R^d_+ $ and $Mz=y\in \Omega_n$ with $M=[N|n]\in \OO(d)$.
By orthogonality of $M$ we have
$$
y \cdot n = Mz \cdot n = z\cdot M^T n = z \cdot e_d = z_d.
$$
Thus if we let $y=y' + (y\cdot n) n$, with $y'\in \partial \Omega_n$, then $Nz' = y'$ for the tangential component.
We now need to derive some bounds on the new
tangential variable $z'$. Observe that $N$ has rank $d-1$, hence $d-1$ of its rows are linearly independent.
Set $N'$ to be a $(d-1)\times (d-1)$ matrix formed by these $d-1$ rows of $N$.
From the overdetermined linear system $Nz' = y'$ we have $N' z' = y''$, where 
$y''\in \R^{d-1}$ is the corresponding part of $y'$. Using the assumption that $|y'|\leq R$,
we get the following bound
\begin{equation}\label{z-est}
|z'|= | (N')^{-1} y'' | \leq c_N |y| \leq c_N R.
\end{equation}
Now if $\Lambda\subset \Z^d$ is the Fourier spectrum of $v_0$, by Remark \ref{rem1} we may assume that
$|N^T \xi| \leq 1/ (8 c_N R )$, for any $\xi \in \Lambda$, from which and (\ref{z-est}) one gets
$$
|\xi \cdot Nz' | = | N^T \xi \cdot z' | \leq \frac{1}{8 c_N R}  c_N R = \frac 18.
$$
The last expression shows that $\cos (  2\pi \xi \cdot Nz' ) \geq \sqrt{2}/2$ for all $\xi \in \Lambda$ and any $z'$
satisfying \eqref{z-est}. By construction of Proposition \ref{Thm-slow-conv}, $\Lambda$ is symmetric with respect to the origin,
also $c_\xi (v_0) = c_{-\xi}(v_0)$ for any $\xi \in \Lambda$, and all non-zero coefficients of $V$ are positive and do not exceed 1. 
Hence for any $y'\in \partial \Omega_n \cap B(0,R)$ we get 
\begin{multline}
v( y' + (y\cdot n) n) = V(Nz', z_d) = \sum\limits_{ \xi \in \Lambda } c_\xi(v_0) e^{  -2\pi |N^T \xi| z_d } e^{ 2\pi i \xi \cdot Nz' }= \\
\sum\limits_{\xi \in \Lambda} c_\xi(v_0) e^{  -2\pi |N^T \xi| z_d } \cos( 2\pi  \xi \cdot Nz' ) 
\geq \frac{1}{\sqrt{2}} \mathcal{S}(z_d; v_0), 
\end{multline}
where $\mathcal{S}$ is defined by \eqref{Parseval}. Finally, recall that $y \cdot n = z_d$, and hence
in the last inequality restricting $z_d  $ to the sequence $\{t_k\}_{k=1}^\infty$, and using the estimate of Proposition \ref{Thm-slow-conv}
we complete the proof of Theorem \ref{Thm-slow-conv-Laplace}.
\end{proof}

It follows from the proof of Theorem \ref{Thm-slow-conv-Laplace} that we may have local uniformity for slow convergence
with respect to tangential directions, meaning that the sequence on which the convergence is slow,
once chosen for the modulus of continuity $\omega$, can be used for any $R>0$.
The only difference is that in this case one should start at a very large index (depending on $R$) in the sequence.


\section{Variable coefficients}\label{sec-variable-coeff}

In this section we prove Theorem \ref{Thm-slow-variable} for coefficient matrix $A$ satisfying (\ref{cond-on-A}).
Recall, that the previous section established the slow convergence phenomenon for Laplace operator.
The main (and big) point that makes Laplacian special in the analysis is
that one may write the solution to reduced problems (\ref{V-of-u})
explicitly. However, for variable coefficient case one does not possess
explicit forms for the solutions, which necessitates a rather different approach. 

We start with some preliminary set-up.
For coefficient matrix $A$ by $A^*$ denote the coefficient matrix of the adjoint operator, i.e. $A^{*, \alpha \beta} = A^{\beta \alpha}$.
For $1\leq \gamma \leq d$ we let $  \chi^{*, \gamma}$ be the smooth solution to the following cell-problem
\begin{equation}\label{cell-problem}
\begin{cases} -\nabla_y \cdot A^*(y) \nabla_y  \chi^{*, \gamma }(y)  = \partial_{y_\alpha} A^{*, \alpha \gamma }    ,&\text{  $ y  \in  \mathbb{T}^d $}, \\
 \int_{ \mathbb{T}^d }  \chi^{*, \gamma} (y) dy =  0   .    \end{cases}
\end{equation}
Next, by $v_n^{*,\gamma}$ we denote the solution, in a sense of Theorem \ref{Thm-Prange}, to the boundary layer problem
\begin{equation}\label{bdry-layer-system-for-v-star}
\begin{cases} 
- \nabla \cdot A(y) \nabla v_n^{*,\gamma} (y) =0 ,&\text{  $y \in \Omega_n$}, \\
 v_n^{*,\gamma}(y)=-\chi^{*,\gamma}(y) ,&\text{  $ y \in \partial \Omega_n  $.}    \end{cases}
\end{equation}
Finally we recall the notion of the Green's kernel.
For a coefficient matrix $A$ and a halfspace $\Omega \subset \R^d$,
the Green's kernel $G=G(y,\yy)$ corresponding to the operator $-\nabla \cdot  A(y) \nabla$ in domain $\Omega$
is a function satisfying the following elliptic equation
\begin{equation}\label{Green-main-half-space}
\begin{cases} -\nabla_{y} \cdot A(y)  \nabla_{y} G(y, \yy) =\delta(y - \yy) ,&\text{  $y \in \Omega $}, \\
 G(y, \yy) =0 ,&\text{  $ y \in \partial \Omega  $,}    \end{cases}
\end{equation}
for any $\yy\in \Omega$, where $\delta$ is the Dirac distribution.
To have a quick reference to this situation, we will say that $G$ is the Green's kernel for \emph{the pair} $(A, \Omega)$.
Note, that the definition of $G$ does not require $A$ to be periodic.
The existence and uniqueness of the Green's kernel for divergence type elliptic systems in halfspaces is proved
in \cite[Theorem 5.4]{Hofmann-Kim} for $d\geq 3$, and for 2-dimensional case in \cite[Theorem 2.21]{Dong-Kim}.
Here we will only use the case of scalar equations.
It is also shown that if $G^*$ is the Green's kernel for the pair $(A^*, \Omega)$ then one has the symmetry relation
\begin{equation}\label{symm-of-Green-transp}
 G(y,\yy)  = G^*(\yy, y), \qquad y,\yy \in \Omega \  \text{ with } y\neq \yy.
\end{equation}
From here we see that Green's kernel has zero boundary values with respect to both variables.

For a unit vector $n \notin \R \Q^d$ and a smooth function $v_0$ let
$v$ be the solution to (\ref{bdry-layer-system-recall}) in a sense of Theorem \ref{Thm-Prange}.
Set $\lambda: = y \cdot n$ for $y\in \Omega_n$, and let $M\in \OO(d)$
satisfy $M e_d = n$. Then for any $0<\kappa < 1/(2d)$ we have
\begin{multline}\label{bdry-layer-sol-repr1}
v(y) = \int_{ \partial \R^d_+} \partial_{2,\alpha} G^n (n, Mz) \times \bigg[ A^{\beta \alpha} (\lambda M z ) n_\beta + \\
 A^{\beta \gamma} (\lambda Mz) \left( \partial_{y_\beta} \chi^{*, \alpha} (\lambda Mz)  
 +  \partial_{y_\beta} v_n^{*,\alpha} ( \lambda M z) \right) n_\gamma \bigg] v_0(\lambda M z) d \sigma(z) +
  O(\lambda^{-\kappa}),
\end{multline}
where $G^n$ is the Green's kernel for the pair $(A^0, \Omega_n)$, and $A^0$ denotes
the matrix of coefficients of the homogenized operator corresponding to $-\nabla_y \cdot A(y) \nabla_y$.
Also $\partial_{2, \alpha} $ denotes differentiation with respect to the $\alpha$-th coordinate of the second variable of $G^n$,
and the error term $O(\lambda^{-\kappa})$ is locally uniform in tangential variable $y': = y- (y\cdot n)n$,
and is independent of the matrix $M$.
The asymptotic formula (\ref{bdry-layer-sol-repr1}) is proved by Prange in \cite[Section 6]{Prange}
for systems of equations. Here, since we are working with scalar equations, we have a slightly simpler form of it.

We are going to switch from the differentiation in $y$ to $z$-variable.
Since $y=Mz$ and $M$ is orthogonal, it is easy to see that 
\begin{equation}\label{from-y-to-z}
\nabla_y = M \nabla_z.
\end{equation}
Set $\boldsymbol{\chi}^{*,\alpha}(z) : = \chi^{*,\alpha}(Mz)$ and  $\textbf{v}_n^{*,\alpha}(z) := v_n^{*,\alpha}(Mz)$.
On the boundary of $ \R^d_+$, for each $1\leq \alpha \leq d$
we have $\boldsymbol{\chi}^{*,\alpha} + \textbf{v}_n^{*,\alpha} = 0$ by (\ref{bdry-layer-system-for-v-star}),
hence taking into account the relation (\ref{from-y-to-z}) and the fact that $M$ has $n$ as its last column, we obtain
$$
\partial_{y_\beta} ( \chi^{*,\beta} + v_n^{*,\beta} )(y)  = n_\beta \partial_{z_d} (\boldsymbol{\chi}^{*,\beta} + \textbf{v}_n^{*,\beta})(z) .
$$
Since the Green's kernel has zero boundary data,
using (\ref{from-y-to-z}) we have that
if $\partial \Omega_n \ni y =Mz $, with $z\in \partial \R^d_+$ then
\begin{equation}\label{green-deriv-rel}
\partial_{y_\alpha} G^n(n,y)=  \partial_{2,\alpha} G^n (n, Mz) = n_\alpha \partial_{z_d} G^{0,n} (e_d, z),
\end{equation}
where $G^{0,n}$ is the Green's kernel for the pair $(M^T A^0 M, \R^d_+)$. Observe that $G^{0,n}$
depends on the matrix $M$, however for the sake of notation we suppress this dependence
in the formulation. 
For now, dependence of $G^{0,n}$ on $M$ plays no role, but later on we will need to make a specific choice of orthogonal matrices $M$.

Applying these observations in (\ref{bdry-layer-sol-repr1}) and grouping similar terms we obtain
\begin{multline}\label{bdry-layer-sol-repr2}
\ \  v(y)  =  
 \int_{\R^{d-1}} \partial_{2,d} G^{0,n} (e_d, (z',0) ) \times n^T A( \lambda M(z',0) ) n \times  \\
 \bigg[ 1+   n_\alpha \big( \partial_{z_d} \boldsymbol{\chi}^{*, \alpha} (\lambda (z',0))  + 
 \partial_{z_d} \textbf{v}_n^{*,\alpha} ( \lambda  (z',0)) \big) \bigg] \times v_0(\lambda M (z',0) ) d z' +   O(\lambda^{-\kappa}).
\end{multline}

The next lemma concerns a particular class of integrals of type (\ref{bdry-layer-sol-repr2}).

\begin{lem}\label{lem-quasi-per-my}{\normalfont(see \cite[Lemma 2.3]{A})}
Assume $H:\R^d \to \R$ is smooth $\Z^d$-periodic function, $n \notin \R\Q^d$ is a unit vector and $M\in \OO(d)$ 
satisfies $M e_d = n$. Then for $h(z') := H(M(z', 0))$, where $z' \in \R^{d-1}$ and any $F\in L^1(\R^{d-1})$ one has
$$
 \lim\limits_{\lambda \to \infty} \int_{\R^{d-1}} F(z') h(\lambda z') dz' = c_0(H) \int_{\R^{d-1}} F(x) dx, 
$$
where $c_0(H)=\int_{\mathbb{T}^d } H(y) dy$.
\end{lem}

This lemma is proved in \cite{A} for functions admitting a certain type of expansion
into series of exponentials. To obtain the current version, one can take the matrix $T$
in Lemma 2.3 of \cite{A} to be the identity.

In the next statement we collect the necessary information concerning the Green's kernel involved in (\ref{bdry-layer-sol-repr2}).

\begin{lem}\label{Lem-Green-props-for-slow}
For any $n\in \mathbb{S}^{d-1}$ and any $M \in \OO(d)$ satisfying $M e_d = n$ let 
$G^{0,n}_M (z, \zz)$ be the Green's kernel for the pair $(M^T A^0 M, \R^d_+)$.
Set $f_{n,M} (z') := \partial_{2,d} G^{0,n}_M (e_d, (z',0))$, where $z' \in \R^{d-1}$.
Then
\vspace{0.1cm}
\begin{itemize}
 \item[{\normalfont(i)}] $ f_{n,M} \in L^1(\R^{d-1}) $ \  and \ 
 $ \inf\limits_{n, M} \left| \int_{\R^{d-1}} f_{n,M}(z') dz' \right| > 0 $,
\vspace{0.1cm}
\item[{\normalfont(ii)}]
 $\sup\limits_{n,M } || f_{n,M} ||_{L^1(\R^{d-1})}<\infty$ \ and \
$\sup\limits_{ n,M  } \int\limits_{|z'|\geq A} |f_{n,M}(z')| dz' \to 0 \text{ as } A \to \infty$,
\vspace{0.1cm}
\item[{\normalfont(iii)}] $f_{n,M} \in C^1(\R^{d-1})$ \ and \ 
$\sup\limits_{n,M} || \nabla' f_{n,M} ||_{L^1(\R^{d-1})}<\infty$,
where $\nabla'$ is the gradient in $\R^{d-1}$.

\end{itemize}
\end{lem}

\begin{proof} 
Recall that $G^{n}$ is the Green's kernel for the pair $(A^0, \Omega_n)$.
The following bound is proved in \cite[Lemma 2.5]{GM-Acta}
\begin{equation}\label{bound-on-Green-n}
 | G^{n} (y, \yy) | \leq C \frac{ (y \cdot n) (\yy \cdot n)}{|y-\yy|^d} , \qquad y\neq \yy \text{ in } \Omega_n,
\end{equation}
where the constant $C$ is independent of $n$. It is easy to observe (see e.g. \cite[Claim 3.1]{A}) that for any $M\in \OO(d)$
with $Me_d = n$ we have $G^{0,n}_M(z, \zz) = G^n (M^T z, M^T \zz) $ for $z \neq \zz$ in $\R^d_+$. From here
and (\ref{bound-on-Green-n}), along with the orthogonality of $M$ one has
\begin{equation}\label{bound-on-Green-0-n}
 | G^{0,n}_M(z,\zz) | = |G^n (M^T z, M^T \zz) |  \leq C  \frac{ (M^T z \cdot n) (M^T \zz \cdot n) }{|M^T z - M^T \zz|^d } = C \frac{z_d \zz_d}{|z-\zz|^d},
\end{equation}
for all $ z\neq \zz$ in $\R^d_+$, with constant $C$ as in (\ref{bound-on-Green-n});
in particular $C$ is independent of $n$ and $M$. Since $G^{0,n}$ has zero data on $\partial \R^d_+$, from 
(\ref{bound-on-Green-0-n}) one easily infers that 
$$
| f_{n,M} (z') | \leq  \frac{C}{|e_d - (z',0)|^d}, \qquad \forall z' \in \R^{d-1},
$$
which shows that $f_{n,M} \in L^1(\R^{d-1})$ as well as part (ii).
For the second statement of (i) let $P^{0,n}_M(z, \zz) $ be the Poisson kernel for the pair $(M^T A^0 M, \R^d_+)$. 
Then for $z \in \R^d_+$, and $\zz \in \partial \R^d_+$ we have
\begin{align*}
P^{0,n}_M (z, \zz) &= - e_d^T ( M^T A^0 M ) \nabla_{\zz} G^{0,n}_M (z, \zz) \\ &= -( M e_d )^T A^0 (M e_d) \partial_{2,d} G^{0,n}_M (z, \zz) \\ &=
- n^T A^0 n \partial_{\zz_d} G^{0,n}_M (z, \zz).
\end{align*}
The last expression, combined with the fact that\footnote{It is proved in \cite[Sectin 2.2]{GM-Acta} (see also \cite[Section 3.2]{Prange}) that the variational
solution to (\ref{bdry-layer-system-recall}) has an integral representation by Poisson's kernel.
Moreover, as long as the asymptotics of the solutions to (\ref{bdry-layer-system-recall}) far away from the boundary of the hyperplane is not concerned,
there are no restrictions imposed on the normal direction.
Since identical 1 is a solution to (\ref{bdry-layer-system-recall}) for a boundary data identically equal to 1,
we may represent this solution by Poisson's kernel,
which shows that the integral of Poisson's kernel is 1.
} 
$\int_{\partial \R^d_+} P^{0,n}_M (e_d, \zz) d \sigma(\zz) =1$,
along with the ellipticity of $A^0$
completes the proof of the second claim of (i).

Finally, for (iii) observe that since $G^{0,n}_M  $ solves an elliptic equation with constant coefficients,
by standard elliptic regularity we have $f_{n,M} \in C^1(\R^{d-1})$. 
For the growth estimate by 
\cite[V.4.2 Satz 3]{Schulze-Wildenhain} one has
$$
| \nabla' \big( \partial_{2,d} G^{0,n}_M (e_d, (z',0) ) \big) | \leq  \frac{C}{|e_d - (z',0)|^d}, \qquad \forall z' \in \R^{d-1},
$$
where $C$ is independent of the unit vector $n$ and the orthogonal matrix $M$.

The proof of the lemma is complete.
\end{proof}

We now turn to the discussion of the core of averaging process of
(\ref{bdry-layer-sol-repr2}).
Our next result is one of the key lemmas of the current paper.

\begin{lem}\label{Lem-slow-conv-family-of-F}
Let $\Xi$ be any non-empty set of indices 
and assume we are given a family of functions $\mathcal{F}=\{F_i\}_{i\in \Xi}$
with the following properties:
\vspace{0.1cm}
\begin{itemize}
 \item[{\normalfont(a)}] $F \in L^1(\R^{d-1})$ for any $F\in \mathcal{F}$ and  $ \inf\limits_{F \in \mathcal{F}} \left| \int_{\R^{d-1}} F (x) dx \right| >0 $,
\vspace{0.1cm} 
 \item[{\normalfont(b)}] $\sup\limits_{F \in \mathcal{F}} || F ||_{L^1(\R^{d-1})}<\infty$ and
$$
\sup\limits_{ F \in \mathcal{F} } \int_{|x|\geq A} |F(x)| dx \to 0 \text{ as } A \to \infty.
$$
\end{itemize}
Let also $\omega$ be any modulus of continuity
and $S_0 \subset \mathbb{S}^{d-1}$ be any open subset of the sphere. Then, there exists an irrational vector 
$n \in S_0$, an unbounded, and strictly increasing sequence
of positive numbers $\{\lambda_k\}_{k=1}^\infty$ such that
for any $F\in \mathcal{F} $ there is
a function $v_0\in C^\infty(\mathbb{T}^d)$ satisfying
\begin{equation}\label{est-of-lemma-slow}
\left|  \int_{\R^{d-1} } F(x) v_0(\lambda_k M(x,0)) dx - c_0(v_0) \int_{\R^{d-1}} F(x) dx   \right| \geq \omega(\lambda_k), \ \ k=1,2,... ,
\end{equation}
whenever $M\in \OO(d)$ and $M e_d = n$.

Assume in addition to {\normalfont(a)} and {\normalfont(b)} that $\mathcal{F}$ also satisfies
\vspace{0.1cm}
\begin{itemize}
 \item[{\normalfont(c)}] $F\in C^1(\R^{d-1})$ for any $F\in \mathcal{F}$
and $$\sup\limits_{F\in \mathcal{F}} ||\nabla F||_{L^1(\R^{d-1})} <\infty,$$
\end{itemize}
then the function $v_0$ too can be chosen independently of $F$.
\end{lem}

\begin{remark}\label{rem-dense-normals-family-F}
Let us remark that the left-hand side of (\ref{est-of-lemma-slow}) decays as $k \to \infty$ in view of Lemma \ref{lem-quasi-per-my},
therefore the lower bound of the current lemma is non-trivial.
The Lemma shows that under conditions (a) and (b) only, the direction, and the sequence along which convergence is slow can be chosen
uniformly for the entire family $\mathcal{F}$. Moreover,
as will be seen from the proof of Lemma \ref{Lem-slow-conv-family-of-F}, for any $F$ and $G$ from $\mathcal{F}$ their corresponding functions $v_0(F)$ and $v_0(G)$
have equal up to the sign Fourier coefficients.
\end{remark}

\begin{proof}[Proof of Lemma \ref{Lem-slow-conv-family-of-F}]
For the sake of clarity we divide the proof into few steps.

\vspace{0.2cm}
\noindent \textbf{Step 1. Construction of $n$ and $\{\lambda_k\}_{k=1}^\infty$.}
We start by determining a suitable modulus of 
continuity for which we will apply Lemma \ref{Lem-bad-normals}
to get the normal $n$.

For $F \in \mathcal{F}$ let $ \mathcal{I}_F $ be the absolute
value of the integral of $F$ over $\R^{d-1}$, and set $\tau_0 : = \inf\limits_{F \in \mathcal{F}} \mathcal{I}_F   $.
By (a) and (b) we have $0<\tau_0 <\infty$.
It is easy to see using the fact that $\tau_0>0$ and condition (b) that there
exists $A_0>0$ large enough such that
\begin{equation}\label{est1-A-0-family}
\left| \int_{|x|\leq A_0}  F(x) dx \right| \geq 2 \int_{|x|\geq A_0} |F(x) | dx + \frac 12 \mathcal{I}_F,
\end{equation}
for any $F\in \mathcal{F}$. For this choice of $A_0$ denote $\e_F : = \frac{1}{||F||_{L^1(\R^{d-1})}} \left| \int_{|x|\leq A_0} F(x) dx \right| $,
where $F\in \mathcal{F}$. Since $\tau_0>0$ from (\ref{est1-A-0-family}) and condition (b) we have
\begin{equation}\label{e-0-def}
 0<\e_0:= \inf\limits_{F\in \mathcal{F}} \e_F \leq 1.
\end{equation}
We now fix some small constant $\delta_0>0$ such that
\begin{equation}\label{choice-of-delta-0}
 | \cos(t) - 1 |\leq \e_0 /4 \ \text{ for any } t\in \R \text{ with } |t|\leq \delta_0.
\end{equation}
Assume that $n\in \mathbb{S}^{d-1}$ and let $M\in \OO(d)$ be so that $M e_d = n$. We then have $M=[N|n]$ where
$N$ is the matrix formed from the first ($d-1$)-columns of $M$. Observe that for any $\xi\in \Z^d$ and any $x\in \R^{d-1}$
one has $\xi \cdot M(x,0) = x^T N^T \xi$. Therefore if for some $\lambda>0$ we have $2\pi \lambda A_0 |N^T \xi| \leq \delta_0$
then (\ref{est1-A-0-family}) and (\ref{choice-of-delta-0}) imply
\begin{align*}
\numberthis \label{est-with-38} \left| \int_{\R^{d-1}} F(x) \cos [  2\pi  \lambda \xi \cdot M(x,0) ] dx \right|\geq \left| \int_{|x|\leq A_0} F(x) dx \right| &- \\
 \int_{|x|\leq A_0} |F(x)| \times | \cos [2\pi \lambda x^T N^T \xi ] -1 | dx &- \int_{|x|>A_0} |F(x)| dx \geq \frac 38 \mathcal{I}_F,
\end{align*}
for any $F \in \mathcal{F}$.
Define 
$$
\omega_1(t) := \frac{\delta_0}{2\pi A_0} \frac{1}{\omega^{-1} \left( \frac 38  \tau_0 \frac{1}{t^t}  \right)  } , \qquad t\geq 1,
$$
where $\omega^{-1}$ stands for the inverse function of $\omega$. Obviously $\omega_1$ is one-to-one,
continuous, and decreases to 0 as $t\to \infty$. It is also clear that $\omega_1$ is well-defined
for large enough $t$, thus without loss of generality we will assume that $\omega_1$ is defined for all $t\geq 1$.
Applying Lemma \ref{Lem-bad-normals} for $\omega_1$ as a modulus of continuity we obtain $\Lambda\subset \Z^d$, and a unit vector $n\notin \R \Q^d$
such that if $M\in \OO(d)$ is any matrix satisfying $M e_d = n$, then
$$
|N^T \xi |\leq \omega_1(|\xi|), \qquad \forall \xi \in \Lambda,
$$
where, as is customary, $d\times (d-1)$ matrix $N$ is formed from the first ($d-1$)-columns of $M$. 
Following Remark \ref{rem-bad-normals-dense-measure-0} we may assume that $n\in S_0$.
We arrange
elements of $\Lambda$ in increasing order of their norms, thus
$\Lambda = \{ \xi^{(k)}: \ k=1,2,...   \}$, where by construction we have $k\leq |\xi^{(k)}|<|\xi^{(k+1)}|$
for any $k\geq 1$. Moreover, according to Remark \ref{rem-large-gap} we may also assume that for any $k\in \N$ we have
\begin{equation}\label{kappa-gap}
 |\xi^{(k)}| < \varrho |\xi^{(k+1)}|  ,
\end{equation}
where $0<\varrho<1$ is a fixed parameter satisfying
\begin{equation}\label{choice-of-kappa}
2 \sup\limits_{F \in \mathcal{F}}  || F||_{L^1(\R^{d-1})} \frac{\varrho}{1-\varrho} < \frac{3}{16} \tau_0.
\end{equation} 
Note that the supremum here is finite by (b) and is non-zero by (a).
Set
\begin{equation}\label{def-lambda}
 \lambda_k := \omega^{-1} \left( \frac 38 \tau_0 \frac{1}{|\xi^{(k)}|^k } \right), \qquad k=1,2,... \ .
\end{equation}
It is clear that $\lambda_k$ is unbounded and is strictly increasing. 
Observe that $n$, and the sequence $\{\lambda_k \}$ are uniform for the entire family $\mathcal{F}$. 

\vspace{0.2cm}
\noindent \textbf{Step 2. Construction of $v_0$ for fixed $F\in \mathcal{F}$.}
We proceed to construction of the function $v_0\in C^\infty(\mathbb{T}^d)$ for the given $F\in \mathcal{F}$,
for which it is enough to construct the sequence of Fourier coefficients of $v_0$, which
we will denote by $\{c_\xi(v_0)\}_{\xi \in \Z^d}$. 

Let $F\in \mathcal{F}$ be fixed. For $\xi \in \Z^d$ if we have
$ \xi \in \Lambda$ then set $c_\xi(v_0) = c_{-\xi}(v_0) = \e_k(F) |\xi|^{-k} $, where $k\in \N$
is the index of $\xi$ in $\Lambda$ according to the increasing rearrangement made above,
and $\e_k(F) \in \{-1,1 \}$ will be chosen below. It is important to note that this sign is the same for $c_\xi$ and $c_{-\xi}$.
Otherwise, if $ \pm \xi \notin \Lambda$ we let $c_\xi(v_0)=0$. Clearly the sequence $\{c_\xi\}$
decays faster than any polynomial rate in $|\xi|$, hence $v_0 $ is smooth. Also, since $c_{\xi}(v_0)= c_{-\xi}(v_0)$
for any $\xi \in \Z^d$ we have that $v_0$ is real-valued. Observe that $c_0(v_0)=0$ by construction, and expanding $v_0$
into Fourier series we get
\begin{multline}\label{exp-of-F-3}
 \int_{\R^{d-1}} F(x) v_0(\lambda M(x,0)) dx = \sum\limits_{m=1}^\infty \frac{2 \e_m(F)}{ |\xi^{(m)}|^m } 
 \int_{\R^{d-1}} F(x) \cos \big( 2\pi \lambda x^T N^T \xi^{(m)} \big) dx := \\
 \frac{2 \e_k(F)}{ |\xi^{(k)}|^k } \mathcal{I}_k(\lambda) + \Sigma_1(\lambda) + \Sigma_2(\lambda),
\end{multline}
where $k\geq 1$, $\mathcal{I}_k(\lambda):= \int_{\R^{d-1}} F(x) \cos( 2\pi \lambda x^T N^T \xi^{(k)} ) dx$,
$\Sigma_1(\lambda)$ contains the part of sum where $m<k$ and $\Sigma_2(\lambda)$ respectively
sums over the range $m>k$. In view of the construction the sums $\Sigma_i(\lambda)$, $i=1,2$ are real-valued
for any $\lambda$. By the definition of $\lambda_k$, the fact that $|\xi^{(k)}| \geq k $ and that $\omega$ is decreasing we easily see that
$ 2\pi \lambda_k A_0 |N^T \xi^{(k)}| \leq \delta_0 $ for any $k$, hence applying (\ref{est-with-38}) we obtain
\begin{equation}\label{est-with-34}
\frac{2  }{ |\xi^{(k)}|^k } | \mathcal{I}_k(\lambda_k) | \geq \frac{3}{4} \mathcal{I}_F \frac{1  }{ |\xi^{(k)}|^k }.
\end{equation}
On the other hand, by (\ref{kappa-gap}) and (\ref{choice-of-kappa}) we easily get
\begin{equation}\label{est-range2}
 | \Sigma_2(\lambda) | \leq 2|| F||_{L^1(\R^{d-1})} \sum\limits_{m=k+1}^\infty \frac{1}{|\xi^{(m)}|^m} 
 \leq \frac{3}{16} \frac{1}{|\xi^{(k)}|^k} \tau_0,
\end{equation}
for any $\lambda\geq 1$.
We now estimate the contribution of the range $m<k$. The triangle inequality implies
$$
\left| \Sigma_1(\lambda_k) + \frac{2  }{ |\xi^{(k)}|^k }   \mathcal{I}_k(\lambda_k)  \right| + 
\left| \Sigma_1(\lambda_k) - \frac{2  }{ |\xi^{(k)}|^k }   \mathcal{I}_k(\lambda_k)  \right| \geq \frac{4  }{ |\xi^{(k)}|^k }  | \mathcal{I}_k(\lambda_k)  |,
$$
hence at least one of the terms in left-hand side of the last inequality is not less than half of the right-hand side. Taking this into account,
we choose the sign $\e_k $ in order to get the largest term from the left-hand side of the above inequality. 
This choice of $\e_k$, combined with estimates (\ref{est-with-34}) and (\ref{est-range2}), and the definition of $\lambda_k$ given by (\ref{def-lambda}) yields
\begin{equation}
\left| \frac{2 \e_k(F)}{ |\xi^{(k)}|^k } \mathcal{I}_k(\lambda_k) + \Sigma_1(\lambda_k) + \Sigma_2(\lambda_k) \right|
\geq \frac{3}{8} \mathcal{I}_F \frac{1  }{ |\xi^{(k)}|^k } \geq \frac{3}{8} \tau_0 \frac{1  }{ |\xi^{(k)}|^k } =  \omega(\lambda_k),
\end{equation}
for any $k=1,2,...$ . The estimate (\ref{est-of-lemma-slow}) of the lemma obviously follows from the last inequality and (\ref{exp-of-F-3}).

\vspace{0.2cm}
\noindent \textbf{Step 3. Uniform choice of $v_0$.}
Lastly, we turn to the proof of possibility of a uniform choice of $v_0$ under additional condition (c).
There is no loss of generality to assume that 
\begin{equation}\label{assump-on-g}
t \omega(t) \to \infty   \text{ as } t\to \infty ,
\end{equation}
since otherwise we would simply replace $\omega$ by a new modulus of continuity $\widetilde{\omega}$,
where $\widetilde{\omega}(t)\geq \omega(t)$ for all $t\geq 1$ and $\widetilde{\omega}$ satisfies (\ref{assump-on-g}),
by that getting even a slower convergence.
Thus we will take (\ref{assump-on-g}) for granted.
For fixed $F \in \mathcal{F}$, $\xi\in \Z^d \setminus \{0\}$, and $\lambda>0$ set 
\begin{equation}\label{I-lambda-xi}
 I(\lambda; \xi) := \int_{\R^{d-1}} F(x) e^{2 \pi i \lambda x^T N^T \xi} dx.
\end{equation}
Let $1\leq k \leq d-1$ be such that the $k$-th component of the vector $N^T \xi $ is the largest by absolute value.
This choice implies $|(N^T \xi)\cdot e_k | \geq (d-1)^{-1/2} |N^T \xi |$, where $e_k$ is the $k$-th vector of the standard basis of $\R^{d-1}$.
Integrating by parts in $I(\lambda; \xi)$ in the direction of $e_k$ we see that 
\begin{equation}\label{bound-by-int-by-parts}
|I(\lambda; \xi)| \leq \frac{1}{\lambda} \frac{\sqrt{d-1}}{2\pi  |N^T \xi|} \sup\limits_{F\in \mathcal{F}} || \nabla F ||_{L^1(\R^{d-1})},
\end{equation}
where the supremum is finite due to the assumption (c).

By (\ref{def-lambda}) we have $|\xi^{(k)}|^k \omega(\lambda_k) =\frac 38 \tau_0$ for each $k\in \N$. 
Also, since $\lambda_k$ is increasing and unbounded by construction, from 
(\ref{assump-on-g}) we get $\lambda_k \omega(\lambda_k) \to \infty$ as $k \to \infty$.
Hence
\begin{equation}\label{lambda-k-vs-xi-k}
 \frac{\lambda_k}{ |\xi^{(k)}|^k} = : a_k \to \infty \text{ as } k \to \infty.
\end{equation}
We now choose an increasing sequence of integers $(i_k)_{k=1}^\infty$ where $i_1= 1$
and if for $k>1$,  $i_{k-1}$ is chosen, we use (\ref{lambda-k-vs-xi-k}) and take $i_k>i_{k-1}$ so large in order to get
\begin{equation}\label{z2}
\frac{1}{a_{i_k}} \frac{\sqrt{d-1}}{\pi  } \sup\limits_{F \in \mathcal{F}}|| \nabla F ||_{L^1(\R^{d-1})}
\sum\limits_{m=1}^{k-1}  \frac{1}{|\xi^{(i_m)}|^{i_m} } \frac{1}{|N^T \xi^{(i_m)}|} \leq \frac{3}{16} \tau_0.
\end{equation}
Clearly the choice of the sequence $(i_k)$ is independent of a particular $F$ since constants in (\ref{z2})
are uniform for the entire family $\mathcal{F}$. We define $v_0$ through its Fourier coefficients as follows.
If for $\xi\in \Z^d$ we have $\xi = \pm \xi^{ (i_k) }$ for some $k\in \N$ then define $c_\xi(v_0) = c_{-\xi}(v_0) = |\xi|^{-i_k}$,
otherwise, set $c_\xi(v_0)=0$. We have that $v_0$ is uniform for all $F\in \mathcal{F}$.
Observe also, that $v_0$ is simply the function from Step 2 with the difference that its
Fourier spectrum is now supported on the frequencies $\{\pm \xi^{ (i_k) }\}_{k=1}^\infty$
and all non-zero Fourier coefficients are positive. 

We now complete the proof by showing that $n$, $\{\lambda_{i_k}\}_{k=1}^\infty$
and $v_0$ defined above satisfy the Proposition. Plugging $v_0$ into (\ref{exp-of-F-3}),
for each integer $k\geq 1$ we get
\begin{equation}\label{int-of-F-uniform-v0}
\int_{\R^{d-1}} F(x) v_0( \lambda M(x,0) ) dx = \frac{2}{|\xi^{i_k}|^{i_k}} \mathcal{I}_{i_k} (\lambda) +
\Sigma_1(\lambda) + \Sigma_2(\lambda),
\end{equation}
where $\mathcal{I}_{i_k}$, $\Sigma_1$ and $\Sigma_2$ are defined as in (\ref{exp-of-F-3}).
We have
$$
\Sigma_1(\lambda ) = \sum\limits_{m=1}^{k-1} \frac{1}{|\xi^{(i_m)} |^{i_m}  } 
\left[ I(\lambda; \xi^{(i_m)} ) + I(\lambda; -\xi^{(i_m)} ) \right] ,
$$
with $I$ defined from (\ref{I-lambda-xi}). From this we obtain
\begin{multline}\label{est-range1-new}
| \Sigma_1(\lambda_{i_k} )| \stackrel{ (\ref{bound-by-int-by-parts}) }{\leq} \frac{1}{\lambda_{i_k}}
\frac{\sqrt{d-1}}{\pi  } \sup\limits_{F \in \mathcal{F}}|| \nabla F ||_{L^1(\R^{d-1})}
\sum\limits_{m=1}^{k-1}  \frac{1}{|\xi^{(i_m)}|^{i_m} } \frac{1}{|N^T \xi^{(i_m)}|} 
\stackrel{ (\ref{z2}) }{\leq} \\
\frac{1}{\lambda_{i_k}} \frac{3}{16} \tau_0 a_{i_k} 
\stackrel{ (\ref{lambda-k-vs-xi-k}) }{\leq}  \frac{3}{16} \frac{1}{ |\xi^{(i_k)}|^{i_k}  } \tau_0.
\end{multline}

For the subsequence $\{ \xi^{ (i_k)  } \}_{k=1}^\infty$ of $\Lambda$ the analogue of estimate
(\ref{est-range2}) becomes
\begin{multline}\label{est-range2-new}
 |\Sigma_2 ( \lambda) | \leq 2|| F||_{L^1(\R^{d-1})} \sum\limits_{m=k+1} \frac{1}{|\xi^{(i_m)}|^{i_m} } \leq 
 2|| F||_{L^1(\R^{d-1})} \sum\limits_{m=i_{k+1}}^\infty \frac{1}{|\xi^{(m)}|^m} \leq \\ 
 \frac{3}{16} \frac{1}{ |\xi^{ (i_{k+1}-1 ) }|^{ i_{k+1}-1} } \tau_0 \leq \frac{3}{16} 
 \frac{1}{ |\xi^{ (i_k ) }|^{ i_k } } \tau_0,
\end{multline}
where as before, we have used (\ref{kappa-gap}) and (\ref{choice-of-kappa}).
Finally, the lower bound on (\ref{int-of-F-uniform-v0})
of the Lemma follows by replacing $k$ with $i_k$ in (\ref{est-with-34})
and applying estimates (\ref{est-range1-new}) and (\ref{est-range2-new})
to (\ref{int-of-F-uniform-v0}).

The proof is now complete.
\end{proof}

Looking ahead let us remark here, that the importance of uniformity of the choices
in Lemma \ref{Lem-slow-conv-family-of-F} will prove crucial in the applications.
We now include a small modification of the previous lemma
to allow compactly supported functions, as well as shift of the origin in the function $v_0$. This situation emerges 
in applications of Lemma \ref{Lem-slow-conv-family-of-F} to integrals arising from Poisson kernels
corresponding to bounded domains.

\begin{lem}\label{Lem-family-on-compact-supp}
Keeping the notation of Lemma \ref{Lem-slow-conv-family-of-F}, and conditions (a) and (b) in force,
assume in addition that the family $\mathcal{F}$ 
has the following properties:
\vspace{0.1cm}

\begin{itemize}
  \item[{\normalfont(c')}] each $F \in \mathcal{F}$ is supported in some closed ball $\overline{B}_F \subset \R^{d-1} $,
where the set of radii of the balls $B_F$ is bounded away from zero and infinity,
\vspace{0.1cm}

  \item[{\normalfont(d)}]
    $F\in C^1(B_F)$ for any $F \in \mathcal{F}$, and $\sup\limits_{F \in \mathcal{F}} ||\nabla F ||_{L^1(B_F)}<\infty$,
\vspace{0.1cm}

 \item[{\normalfont(e)}] $\sup\limits_{F \in \mathcal{F}} ||  F ||_{L^\infty(B_F)}<\infty$.

\end{itemize}

Then there exist an irrational vector $n\in S_0$, and an unbounded, strictly increasing sequence of positive
numbers $\{\lambda_k\}_{k=1}^\infty$ such that
for any $X_0 \in \R^d$ there exists
a real-valued function $v_0 \in C^\infty(\mathbb{T}^d)$ for which
the estimate
\begin{equation}\label{est-of-lemma-slow-with-shift}
\left|  \int_{\R^{d-1} } F(x) v_0(\lambda_k M(x,0) + \lambda_k X_0 ) dx - c_0(v_0) \int_{\R^{d-1}} F(x) dx   \right| \geq \omega(\lambda_k)
\end{equation}
holds for any $F\in \mathcal{F}$ and each integer $k\geq 1$, whenever $M\in \OO(d)$ and $M e_d = n$.
\end{lem}

\begin{proof}
For notational convenience we extend all functions $F\in \mathcal{F}$ to $\R^{d-1}$ as zero outside their supports.
Observe that here the finiteness of the supremum of part (b) of Lemma \ref{Lem-slow-conv-family-of-F}
is automatically fulfilled.

We will start with the case $X_0=0$. First carry out the proof of Lemma \ref{Lem-slow-conv-family-of-F} up to the definition (\ref{I-lambda-xi}).
Since now functions from $\mathcal{F} $ have compact support, the bound in (\ref{bound-by-int-by-parts})
can not be obtained directly from integration by parts due to boundary terms appearing in the integration.
To overcome this technicality we introduce smooth cut-offs. Let $F \in \mathcal{F}$ be fixed, and let the closed ball 
$\overline{B}=\overline{B}(x_0,r)  $ be the support of $F$. Then by (c') we know that $r \geq c_0>0$ for some absolute
constant $c_0$. For any $\lambda>1/c_0$ we have $B(x_0, r-\lambda^{-1}) \subset B(x_0, r)$ and
 we let $\varphi_\lambda:\R^{d-1}\to [0,1] $ be a smooth function such that $\varphi_\lambda=1$ on $B(x_0, r-\lambda)$,
$\varphi_\lambda = 0 $ on $\R^d\setminus B(x_0, r)$, and $|\nabla \varphi_\lambda(x)| \leq c_1 \lambda$
for any $x\in \R^{d-1}$, where $c_1$ is some absolute constant.

Denoting by $\mathbb{I}_B$ the characteristic function of the ball $B(x_0,r)$,
we decompose (\ref{I-lambda-xi}) into
\begin{multline*}
I(\lambda; \xi ) = \int_{\R^{d-1}} F(x) \mathbb{I}_B (x) \varphi_\lambda (x)  e^{2\pi i \lambda x^T N^T \xi} dx + 
\int_{\R^{d-1}} F(x) \mathbb{I}_B (x) (1-\varphi_\lambda (x))  e^{2\pi i \lambda x^T N^T \xi} dx \\ : =  I_1(\lambda; \xi ) +I_2(\lambda; \xi ) .
\end{multline*}
Observe that in $I_2 (\lambda; \xi )$ we have integration over $B(x_0,r) \setminus B(x_0, r-\lambda^{-1}) $,
 hence by (e)
\begin{equation}\label{bound-on-I-2}
 |I_2 (\lambda; \xi )| \leq  | B(x_0,r) \setminus B(x_0, r-\lambda^{-1}) | \sup\limits_{F \in \mathcal{F}} ||  F ||_{L^\infty(\R^{d-1})}  \leq C \lambda^{-1}.
\end{equation}
For $I_1 (\lambda; \xi)$ we proceed as in (\ref{bound-by-int-by-parts}), however here we will have an additional term
coming from partial integration, namely the one involving $F \partial_k \varphi_\lambda$. But observe that
all derivatives of $\varphi_\lambda$
are supported on $B(x_0,r) \setminus B(x_0, r-\lambda^{-1}) $, and hence using the estimate on the gradient of $\varphi_\lambda$, along with condition (e) we get
$$
\int_{\R^{d-1}} | F(x) \partial_k \varphi_\lambda (x)| dx \leq c_1 \lambda 
| B(x_0,r) \setminus B(x_0, r-\lambda^{-1}) | \sup\limits_{F \in \mathcal{F}} ||  F ||_{L^\infty(\R^{d-1})} \leq C
$$
where constants are uniform in $F$ and $\lambda$. The last bound combined with (\ref{bound-on-I-2}) 
enables us to obtain the estimate in (\ref{bound-by-int-by-parts}) with possibly different absolute constants.
Then, the proof of the current lemma for $X_0=0$ follows from exactly the same argument in
Lemma \ref{Lem-slow-conv-family-of-F} starting from (\ref{bound-by-int-by-parts}) up to the end.

We now consider the case of any $X_0\in \R^d$. Let
$n$ be the normal, $\{\lambda_k\}$ be the sequence, and $\widetilde{v_0}$ be the function
for which  (\ref{est-of-lemma-slow-with-shift}) holds with $X_0=0$.
By constructions of Lemma \ref{Lem-slow-conv-family-of-F} (Step 3 in particular) and 
for the case of $X_0=0$
there is a strictly increasing sequence of integers $\{i_m\}_{m=1}^\infty$
and a set $\widetilde{\Lambda}=\{\xi^{(m)} \}_{m=1}^\infty \subset \Z^d\setminus \{0\}$
satisfying $|\xi^{(m)}|<|\xi^{(m+1)}|$ for all $m \in \N$,
such that 
$$
\widetilde{v_0} (\theta ) = \sum\limits_{m=1}^\infty \frac{1}{|\xi^{ (m) }|^{ i_m  } } 
\left[ e^{ 2\pi i \xi^{(m)} \cdot \theta  } + e^{ - 2\pi i \xi^{(m)} \cdot \theta  } \right], \qquad \theta \in \mathbb{T}^d.
$$
We now slightly adjust the coefficients of $\widetilde{v_0}$ to handle the effect of the shift.
Namely, consider the function
$$
v_0 (\theta ) = \sum\limits_{m=1}^\infty \frac{1}{|\xi^{ (m) }|^{ i_m  } } 
\left[  e^{ - 2\pi i \lambda_m \xi^{(m)} \cdot X_0  }   e^{ 2\pi i \xi^{(m)} \cdot \theta  } +  e^{  2\pi i \lambda_m \xi^{(m)} \cdot X_0  }  e^{ - 2\pi i \xi^{(m)} \cdot \theta  } \right].
$$
By definition, the Fourier coefficient $c_\xi (v_0 )$ is the complex conjugate of $c_{-\xi}(v_0)$ for any $\xi\in \Z^d$,
hence $v_0$ is real-valued. It is also clear that $v_0\in C^\infty(\mathbb{T}^d)$ and $c_0(v_0)=0$.
Following (\ref{exp-of-F-3}) and plugging $v_0$ into (\ref{est-of-lemma-slow-with-shift})
for each integer $k\geq 1 $ we get
$$
 \int\limits_{\R^{d-1} } F(x) v_0(\lambda_k M(x,0) + \lambda_k X_0 ) dx = 
 \frac{2}{| \xi^{ (k) } |^{i_k} } \int\limits_{\R^{d-1}} F(x) \cos ( 2\pi \lambda_k x^T N^T \xi^{(k)} ) dx + \Sigma_1(\lambda_k) + \Sigma_2(\lambda_k),
$$
where $\Sigma_1$ and $\Sigma_2$ are defined in analogy with (\ref{exp-of-F-3}).
Observe that the integral on the right-hand side of the above equality
is the same as for $X_0=0$,
and the sums in $\Sigma_1$ and $\Sigma_2$
can be estimated exactly as in the case $X_0=0$.
Indeed, the only difference is that coefficients in the sums are
multiplied by complex numbers having length 1
(namely the exponents involving $X_0$).
This fact will have no effect when taking absolute values of the terms
in the sums, which is precisely what we
do to bound $\Sigma_1$ and $\Sigma_2$.
Since the analysis is reduced to the case $X_0=0$,
the proof of the lemma is now complete.
\end{proof}

\begin{proof}[Proof of Theorem \ref{Thm-slow-variable}]
We will assume that
\begin{equation}\label{omega-decay-in-thm}
t^{\frac{1}{4d}} \omega (t) \to \infty, \ \ \text{ as } t\to \infty.
\end{equation}
This assumption results in no loss of generality, for a similar argument as in (\ref{assump-on-g}).
The reason for (\ref{omega-decay-in-thm})
is to have slower speed of decay than the error term involved in (\ref{bdry-layer-sol-repr2})
for a parameter $\kappa=1/(4d)$.

Let $1\leq \gamma\leq d$ be fixed from (\ref{cond-on-A}).
From (\ref{cell-problem}) and (\ref{cond-on-A})
we get $\chi^{*,\gamma} = 0$. The latter combined with (\ref{bdry-layer-system-for-v-star}) implies
that for the corresponding boundary layer corrector
we have $v_n^{*,\gamma} = 0$ for any $n\in \mathbb{S}^{d-1}$. 
For $1\leq \alpha \leq d$ and $n\in \mathbb{S}^{d-1}$, $v_n^{*,\alpha}$ solves a uniformly elliptic
PDE in $\Omega_n$ with periodic and smooth coefficients,
and with boundary data $\chi^{*,\alpha}$, hence standard elliptic regularity implies that
there is a constant $C_0$ independent of $n$, such that $|\nabla_y v_n^{*,\alpha} (y) | \leq C_0$
for any $y\in \partial \Omega_n$ and each $1\leq \alpha \leq d$. From this, and the fact that $v_n^{*,\gamma}=0$ it follows that
there exists an open subset of the sphere $\mathbb{S}_\gamma \subset \mathbb{S}^{d-1}$, such that
for any $n\in \mathbb{S}_\gamma $ and any $M\in \OO(d)$ with $M e_d = n$, one has
\begin{equation}\label{Z15}
\bigg| 1+   n_\alpha \big[ \partial_{z_d} \boldsymbol{\chi}^{*, \alpha} (\lambda  (z',0))  + 
 \partial_{z_d} \textbf{v}_n^{*,\alpha} ( \lambda   (z',0)) \big] \bigg| \geq \frac 12,
\end{equation}
for all $z'\in \R^{d-1}$ and any $\lambda>0$. Indeed, we simply choose $\mathbb{S}_\gamma$ so that each $n\in \mathbb{S}_\gamma$
has its $\gamma$-th component sufficiently close to 1.

For $n\in \mathbb{S}^{d-1}$ we fix some $M_n \in \OO(d)$ satisfying $M_n e_d = n$ and let $G^{0,n} (\cdot, \cdot )$
 be the Green's kernel for the pair $(M^T_n A^0 M_n, \R^d_+)$.
Consider the family of functions $\mathcal{F}:=\{F_n\}_{n\in \mathbb{S}^{d-1}}$,
where we have $F_n (x):= \partial_{2,d} G^{0,n} (e_d, (x,0)) $ for $x\in \R^{d-1}$.
By Lemma \ref{Lem-Green-props-for-slow} the family $\mathcal{F}$ satisfies conditions
(a), (b) and (c) of Lemma \ref{Lem-slow-conv-family-of-F}.

Recall that solutions to boundary layer problems are constructed
via the reduced boundary layer systems of form (\ref{bl-system-d+1}), hence we have
\begin{equation}\label{Z16}
 v_n^{*,\alpha} (y) = v_n^{*,\alpha} (M_n z)  = \textbf{v}_n^{*,\alpha} (z)  = V_n^\alpha (N_n z', z_d)
\end{equation}
where $V_n^\alpha$ solves the corresponding problem (\ref{bl-system-d+1}), and as is usual $M_n = [N_n| n]$.
In particular we have that $V_n^\alpha(\cdot, t)$ is $\Z^d$-periodic for any $t\geq 0$,
 and is smooth with respect to all its variables.
Here one should take into account the subtlety, that $V_n^\alpha$, and hence also $\textbf{v}_n^{*,\alpha}$,
implicitly depend on the matrix $M_n$, but as the choice of $M_n$ is now fixed,
we may ignore this dependence.

For $n\in \mathbb{S}^{d-1}$ consider the function
$$
\Psi_n(y) := 1+ n_\alpha \big( n\cdot \nabla_y \chi^\alpha (y) +   \partial_t V_n^\alpha (y,0) \big), \qquad y \in \R^d,
$$
where $V_n^\alpha$ is fixed from (\ref{Z16}). Clearly, $\Psi_n \in C^\infty(\mathbb{T}^d)$. 
From (\ref{from-y-to-z}) we see that $\partial_{z_d} = n \cdot \nabla_y$
which gives the relation between normal derivatives.
Now, if $y\in \partial \Omega_n$,
we get
\begin{equation}\label{V-and-normal-deriv}
 \partial_t V_n^\alpha( N z', 0)  = 
 \partial_{z_d} \textbf{v}_n^{*,\alpha}(z,0)= n\cdot \nabla_y v_n^{*,\alpha}(y).
\end{equation}
From here, the definition of $\Psi_n$ and (\ref{Z15}), let us show that for any \emph{irrational} $n\in \mathbb{S}_\gamma$ one has 
\begin{equation}\label{psi-is-large}
| \Psi_n (y)| \geq 1/2  \qquad \text{ for all } \qquad y \in  \R^d.
\end{equation}
The small nuance, that (\ref{psi-is-large}) needs the normal to be irrational as compared with (\ref{Z15})
lies in the fact that (\ref{Z15}) holds on the boundary of $\Omega_n$, while here we need the entire space $\R^d$.
To see (\ref{psi-is-large}), observe that for $y=Nz'$ with $z'\in \R^{d-1}$ the lower bound we need is due to (\ref{Z15})
and (\ref{V-and-normal-deriv}). Now, if the normal $n$ is irrational, then $\{Nz' : \ z'\in \R^{d-1} \}$
is everywhere dense in $\mathbb{T}^d$, which is the unit cell of periodicity of $\Psi_n$,
hence the continuity of $\Psi_n$ completes the proof of (\ref{psi-is-large}).

We now apply Lemma \ref{Lem-slow-conv-family-of-F} for the family $\mathcal{F}$ and modulus of continuity $2\omega$,
and let $\nu \in \mathbb{S}_\gamma $ be the unit irrational vector and $\{\lambda_k\}_{k=1}^\infty$ be the increasing sequence given by Lemma \ref{Lem-slow-conv-family-of-F}.
Thus for $\nu$
we have (\ref{psi-is-large}). 
Next, for a function $F_\nu(x) $ let $\widetilde{v}_0\in C^\infty(\mathbb{T}^d)$ be the function given by Lemma \ref{Lem-slow-conv-family-of-F}
for which
\begin{equation}\label{Z13}
\left| \int_{\R^{d-1}} F_\nu(x) \widetilde{v}_0(\lambda_k M_\nu (x,0)) dx  - c_0(\widetilde{v}_0 ) \int_{\R^{d-1}} F_\nu(x) dx \right| \geq 2\omega(\lambda_k),
\end{equation}
where $c_0(\widetilde{v}_0 )$ is the $0$-th Fourier coefficient and $k=1,2,...$ .
Ellipticity of $A$ implies that $\nu^T A(y) \nu \geq c_0 |\nu|^2=c_0 $ for any $y\in \R^d$, with absolute constant $c_0>0$,
hence taking into account (\ref{psi-is-large}) we define
\begin{equation}\label{Z14}
v_0(y) := \frac{1}{\nu^T A(y) \nu} \frac{1}{\Psi_\nu(y)}  \widetilde{v}_0 (y) , \qquad y\in \R^d,
\end{equation}
and get $v_0\in C^\infty(\mathbb{T}^d)$.

Finally, we claim that $\nu$, $\{\lambda_k\}_{k=1}^\infty$, and $v_0$ defined by (\ref{Z14}) satisfy the Theorem. Indeed
by (\ref{bdry-layer-sol-repr2}) the solution to boundary layer problem with these parameters has the form
$$
v(y) = \int_{\R^{d-1}} \partial_{2,d} G^{0,\nu} (e_d, (z',0)) \widetilde{v}_0(\lambda M_\nu(z',0)) dz' +O(\lambda^{-\frac{1}{4d}}),
$$
where the parameter $\kappa$ in (\ref{bdry-layer-sol-repr2}) is set to $1/(4d)$.
The last expression combined with (\ref{Z13}) and (\ref{omega-decay-in-thm})
completes the proof of the Theorem.
\end{proof}

\section{Application to boundary value homogenization}\label{sec-application-to-Dirichlet}

This section is devoted to the proof of Theorem \ref{Thm-slow-Dirichlet},
but before delving into details, we sketch the main idea behind the proof.
By \cite{ASS1}-\cite{ASS3} we know that the boundary value homogenization
of type considered in (\ref{Dirichlet-bdd-domains})
is determined by geometric properties of the boundary of the reference domain,
such as non-vanishing Gaussian curvature \cite{ASS1}-\cite{ASS2}, or flat pieces with Diophantine normals
\cite{ASS3}.
Under these conditions one is able to deduce effective upper bounds on convergence rates for the homogenization,
where the rate will be uniform with respect to the boundary data.
With these in mind, a suitable candidate of domain $D$ for Theorem \ref{Thm-slow-Dirichlet}
is a $C^\infty $ domain such that part of its boundary has non-vanishing Gaussian curvature,
while the rest is a piece of a hyperplane. 
Then relying on integral representation of solutions to (\ref{Dirichlet-bdd-domains}) 
via Poisson kernel,
one splits the integral into two parts, namely over curved and flat boundaries.
The next step is to show that the contribution of the curved part has
a prescribed rate of decay determined only by the embedding of $\partial D$ into $\R^d$,
and hence is invariant under rotations of the domain. This step can be fulfilled
by adapting the methods of \cite{ASS1}-\cite{ASS2}.
For the integral over the flat part one shows, using Lemma \ref{Lem-family-on-compact-supp}, that
after a suitable rotation of the domain and an appropriate choice
of the boundary data it
can be made comparatively large. In this section we rigorously implement this idea.

\subsection{Preliminary results}\label{subsec-prelim}
We present some technical results which will be used for the proof of Theorem \ref{Thm-slow-Dirichlet}
below.

Assume we have a bounded domain $D\subset \R^d$ ($d\geq 2$) with smooth boundary,
and a divergence form operator $\mathcal{L}:=-\nabla \cdot A(x) \nabla $
where coefficient matrix $A$ is defined in $\overline{D}$, and is strictly elliptic,
and smooth. Note, that we do not impose any structural condition nor any periodicity assumption on $A$.
Next, we let $P(x,y):D\times \partial D \to \R$ be the Poisson kernel for the operator $\mathcal{L}$ in the domain $D$.
Then by Lemma \ref{Lem-Poisson-dist-smooth} we have
\begin{equation}\label{Poisson-dist-1}
 |P(x,y)|\leq C_P \frac{d(x)}{|x-y|^d},\qquad x\in D, \ y\in \partial D,
\end{equation}
where the constant $C_P=C_P(A, D, d)$. With this notation we have

\begin{lem}\label{lem-Poisson-on-a-portion}
Let $\Pi$ be an open and connected subset of $\partial D$
and $ D_0 \Subset D  $ be fixed.
Then  
$$
  \inf\limits_{x\in D_0} \int_{\Pi} P(x, y)d\sigma(y) >0.
$$
\end{lem}

\begin{proof}
As a trivial observation, before we start the proof, notice that if $\Pi$ is the entire boundary of $D$,
then the integral in question is identically
1. The general case, however, requires some care.
The proof is motivated by \cite[Lemma 3.1]{ASS2}.
We will use integral representation of solutions, to get a more precise
version of the maximum principle.
 Fix a sequence of smooth functions $\{g_n\}_{n=1}^\infty$,
where $g_n:\R^d \to [0,1]$ such that
$g_n=1$ on $\Pi$, and for any domain $\widetilde{\Pi}\subset \R^d$ which compactly contains $\Pi$, 
the sequence $\{g_n\}$ uniformly converges to 0 outside $\widetilde{\Pi}$
as $n\to \infty$. We let $u_n$ be the solution to Dirichlet problem for $\mathcal{L}$ in domain $D$ having boundary
data $g_n \big|_{\partial D}$.
Fix some $\xi\in \Pi$. Since $\Pi $ is open in $\partial D$ there exists $\delta_0>0$ such that
\begin{equation}\label{delta0}
 \{y\in \partial D: | y-\xi |\leq \delta_0   \} \subset \Pi.
\end{equation}
We get
\begin{align*}
\numberthis \label{u-n-minus-g-n} |u_n(x) - g_n(\xi)| = 
\left| \int_{\partial D} P(x,y) [g_n(y) - g_n(\xi) ] d\sigma(y) \right| \leq& \\
\int_{\partial D} \big| P(x,y) [g_n(y) - g_n(\xi) ] \big| d\sigma(y) =& \\
\int_{|y-\xi|> \delta_0 }\big| P(x,y) [g_n(y) - g_n(\xi) ] \big| d\sigma(y) \leq  
2C_P d(x) &||g_n ||_{L^\infty} \int_{|y-\xi|> \delta_0 } \frac{d \sigma(y)}{|x-y|^d}, 
\end{align*}
where we have used (\ref{delta0}) and the fact that $g_n=1$ on $\Pi$ to pass from the second row of (\ref{u-n-minus-g-n})
to the first expression of the last row. We now choose $x\in D$ such that $|x-\xi|<\delta_0 /2$.
The triangle inequality implies $|x-y|\geq \delta_0/2$ for all $y\in \partial D$ satisfying $|y-\xi|>\delta_0$.
Hence, the last integral in (\ref{u-n-minus-g-n}) can be estimated as follows
\begin{equation}\label{a1}
\int_{|y-\xi|> \delta_0 } \frac{d \sigma(y)}{|x-y|^d} \leq 2^d \int_{|y-\xi|> \delta_0 } \frac{d \sigma(y)}{|y-\xi|^d} \\
\lesssim \int_{\delta_0}^1 \frac{t^{d-2}}{t^d} dt \leq C_0 \frac{1}{\delta_0},
\end{equation}
with some positive $C_0=C_0(d)$ uniform in $x$ and $\delta_0$,
and we have invoked integration in spherical coordinates to estimate the surface integral.
Without loss of generality we assume that constants $C_0, C_P \geq 1$.
We now fix $x_0\in D$ such that 
\begin{equation}\label{a2}
|x_0-\xi|\leq \frac{\delta_0}{2} \qquad  \text{ and } \qquad   \frac{\delta_0}{10 C_0 C_P} \leq d(x_0) \leq \frac{\delta_0}{4 C_0 C_P} .
\end{equation}
This is always possible, provided $\delta_0 >0$ is small enough.
Denote $\widetilde{D}:= D_0 \cup \{x_0\}$.
From (\ref{a1}) and (\ref{u-n-minus-g-n}) we obtain
$|u_n(x_0) - g_n(\xi)|   \leq 1/2 $, hence the triangle inequality implies
\begin{equation}\label{a3}
|u_n(x_0)|= |g_n(\xi) +u_n(x_0)-g_n(\xi)|\geq |g_n(\xi)| - |u_n(x_0)-g_n(\xi)| \geq 1-1/2 =1/2.
\end{equation}
From (\ref{a3})
one has $\sup_{\widetilde{D}} u_n \geq 1/2  $ for any $n\in \N$. From the maximum principle we infer that all $u_n$
are everywhere non-negative in $D$, thus applying Moser's version of Harnack's inequality
(see e.g. \cite[Theorem 8.20]{GT})
we get
\begin{equation}\label{Harnack-Mozer}
1/2\leq \sup_{\widetilde{D}} u_n \leq C_1 \inf_{ \widetilde{D} } u_n \leq C_1 \inf_{ D } u_n, \qquad n=1,2,... ,
\end{equation}
where the constant $C_1=C_1(d, A, \widetilde{D}, D)$.
We have the representation
$$
u_n(x) = \int_{\partial D} P(x,y) g_n(y) d\sigma(y), \qquad x\in D.
$$
Using the construction of $g_n$ we pass to the limit in the last integral, getting by (\ref{Harnack-Mozer}) that
$$
\int_{\Pi} P(x,y) d\sigma(y) =\lim\limits_{n\to \infty} u_n(x) \geq \frac{1}{2 C_1}, \qquad \forall x\in D_0.
$$

The proof is complete.
\end{proof}

We will also need a version of the last lemma for a family
of Poisson kernels corresponding to a rotated images of a given domain.
To fix the ideas, recall that the coefficient matrix of (\ref{Dirichlet-bdd-domains}) is defined
on some fixed domain $X$. Let $D\Subset X$ be a bounded domain with $C^\infty$ boundary.
For a matrix $M\in \OO(d)$
define an orthogonal transformation $\mathcal{M}: \R^d \to \R^d$ by $\mathcal{M}x = Mx$, $x\in \R^d$
and consider the rotated domain $D_M := \mathcal{M} D $.
Obviously $D_M$ is a bounded domain, it is also clear that $\partial D_M =\mathcal{M}(\partial D_0) $
since $\mathcal{M}$ is a diffeomorphism. Next, the smoothness of $D_M$ follows directly form the definition
of smooth boundary (see  e.g. Section 6.2 of \cite{GT}).
We shall take the diameter of $D$ sufficiently small so that $D_M\Subset X$ for any $M\in \OO(d)$.
Now we get the following version of the previous lemma.

\begin{cor}\label{cor-Poisson-inf-rotated}
Let $D\subset X$ be as above and let $\Pi \subset \partial D$ be open and connected.
Fix $B_0\Subset D$ and for some small constant $c_0>0$ denote 
$$
\widetilde{\OO} : = \{ M\in \OO(d): \  B_0\subset D_M \text{ and }  \mathrm{dist}(B_0, \partial D_M)\geq c_0  \}.
$$
Then
$$
\inf\limits_{M\in \widetilde{\OO}, \ x\in B_0 } \int_{\Pi_M} P_M(x,y) d\sigma(y) >0,
$$
where $\Pi_M = \mathcal{M}(\Pi)$ and $P_M$ is the Poisson kernel for the pair $(A, D_M)$.
\end{cor}

\begin{proof}
For each fixed $M\in \widetilde{\OO}$ the infimum is positive
in view of Lemma \ref{lem-Poisson-on-a-portion}, and
we will simply follow the dependence of constants in the proof of Lemma \ref{lem-Poisson-on-a-portion}
on the rotation introduced by $M$. First of all by Lemma \ref{Lem-Poisson-dist-smooth}
we have that the constant in (\ref{Poisson-dist-1}) is independent of $M$
therefore (\ref{a2}) is uniform with respect to $M$.
Concerning the use of Harnack inequality in (\ref{Harnack-Mozer})
referring to \cite[Theorem 8.20]{GT}
we know that constants for  a ball of radius $R>0$
depend on dimension of the space, ellipticity bounds of the operator, and the radius $R$. 
Here we have the same operator for all domains $D_M$.
Finally, the uniform distance of $B_0$ from the boundary of $D_M$ for each $M\in \widetilde{\OO}$,
along with standard covering argument extending the Harnack inequality
from balls to arbitrary sets, shows that the constant $C_1$
of (\ref{Harnack-Mozer}) can be chosen independently of $M$.
As the choice of all constants in the proof of Lemma \ref{lem-Poisson-on-a-portion}
can be made uniform with respect to $M\in \widetilde{\OO}$,
the proof is complete.
\end{proof}

The next lemma is used in the localization argument of Proposition \ref{int-in-prop} below.

\begin{lem}\label{Lem-Hessian-one-to-one}
For $r_0>0$ let $\psi \in C^3(\overline{B_{r_0}(0)})$ and assume that
$\psi(0) = |\nabla \psi(0)|=0$ and 
$$
\mathrm{Hess} \psi(0) : = (\partial_{\alpha\beta}^2 \psi)_{\alpha, \beta =1}^d (0)= \mathrm{diag}(a_1,...,a_d)\in M_d(\R), 
$$
where $0<a_1 \leq ...\leq a_d$. Then there exist positive constants 
$c_0=c_0(d, ||\psi ||_{C^3}) <  c_1 = c_1(d, ||\psi ||_{C^3}  )<  r_0/a_1 $,
and 
$K_1=K_1(d)\leq 1  \leq K_2=K_2(d, ||\psi ||_{C^2})$ such that
\vspace{0.2cm}
\begin{itemize}
\item[\normalfont{(i)}]
$
K_1 a_1 |x-y| \leq | \nabla \psi(x) - \nabla \psi(y) | \leq K_2 |x-y|, \qquad \forall x,y\in B(0,c_1 a_1),
$
\vspace{0.2cm}
\item[\normalfont{(ii)}]
if $\delta_{\alpha \beta}$ is the Kronecker symbol, then for any $1\leq \alpha, \beta \leq d$ one has
$$
\left| \partial^2_{\alpha \beta} \psi(x) - a_\alpha \delta_{\alpha \beta}  \right|\leq \frac{a_1}{20 d}, \qquad 
\forall x\in B(0, c_1 a_1),
$$

\item[\normalfont{(iii)}]
$
B(0, c_0 a_1^2) \subset (\nabla \psi)( B(0,c_1 a_1)).
$

\end{itemize}

\end{lem}

\begin{proof}
We start with part (i). For any $x, y \in B_{r_0}(0)$ by Mean-Value Theorem we have
\begin{equation}\label{upper-bound-grad-of-psi-1}
| \nabla \psi(x) - \nabla  \psi(y) | \leq 
\sum\limits_{\alpha=1}^d |\partial_\alpha \psi(x ) - \partial_\alpha \psi(y)| \leq 
\sum\limits_{\alpha,\beta=1}^d  || \partial_{\alpha \beta}^2 \psi||_{L^\infty( B_{r_0}(0))}  |x-y| ,
\end{equation}
which demonstrates the upper bound of (i).
To obtain a lower bound,
for $1\leq \alpha \leq d$ and $x \in B_{r_0}(0)$ set
$g_\alpha (x) = |\partial_{\alpha \alpha}^2 \psi (x) |
  - \sum\limits_{\beta=1, \beta \neq \alpha }^d |\partial_{\alpha \beta}^2 \psi(x)| $,
obviously $g_\alpha(0)= a_\alpha >0$. 
By $C^3$-smoothness of $\psi$ we have that each $g_\alpha$ has linear modulus of continuity, hence
there exists a constant $c_1=c_1(d, ||\psi ||_{C^3} ) $ such that for all $|x|\leq c_1 g_\alpha(0)$ we get
$g_\alpha(x) >  g_\alpha (0)/2$.
From here for any $x, y \in B_{c_1 a_1} (0)$ Mean-Value Theorem yields
\begin{multline}\label{lower-bound-grad-of-psi-1}
| \nabla \psi(x ) - \nabla \psi(y ) | \geq 
c_d \sum\limits_{\alpha=1}^d | \partial_\alpha \psi(x) - \partial_\alpha \psi(y) |
=c_d \sum\limits_{\alpha=1}^d | \nabla (\partial_\alpha \psi) (\tau_\alpha) \cdot (x-y) |,
\end{multline}
where $\tau_\alpha$ lies on the segment $[x,y]$
and $c_d$ is a constant depending on dimension.
We next fix $1\leq \alpha \leq d$ from $|x_\alpha - y_\alpha| = \max\limits_{1\leq \beta \leq d}|x_\beta - y_\beta|$,
and invoking (\ref{lower-bound-grad-of-psi-1}) we get
\begin{equation}\label{lower-bound-grad-of-psi-2}
| \nabla \psi(x ) - \nabla  \psi(y ) | \geq c_d |\nabla  (\partial_\alpha \psi)(\tau_\alpha) \cdot (x-y)| \geq \\
c_d g_\alpha (\tau_\alpha) |x_\alpha - y_\alpha| \geq c_d a_1 |x-y| .
\end{equation}
Combining (\ref{upper-bound-grad-of-psi-1}) and (\ref{lower-bound-grad-of-psi-2})
for any $x,y\in B_{c_1 a_1}(0)$ we obtain
\begin{equation}\label{bi-Lip-on-grad}
K_1 a_1 |x-y| \leq |\nabla \psi(x) - \nabla \psi(y) | \leq K_2 |x-y|, 
\end{equation}
with constants $K_2 = K_2(d, || \psi||_{C^2}) \geq 1$ and $K_1=K_1(d)\leq 1$.
 This completes the proof of part (i).

\vspace{0.2cm}

Since $\partial^2_{\alpha \beta} \psi$ for each $1\leq \alpha,\beta\leq d$ has linear modulus of continuity, 
the claim of (ii) follows easily by Mean-Value Theorem
with $c_1$ sufficiently small. We now proceed to (iii).

\vspace{0.2cm}

By (\ref{bi-Lip-on-grad}) the mapping $\nabla \psi$
is invertible in a neighbourhood of the origin, and 
(iii) is simply an effective version of Inverse Mapping Theorem.
The desired bound follows from the estimate $|| ( \mathrm{Hess} \psi(0))^{-1} ||\asymp a_1^{-1} $,
$C^3$-smoothness of $\psi$, and 
\cite[Theorem 1.1]{Wang}\footnote{Observe that the function $\psi(x) = a_1x_1^2+...+a_d x_d^2$
manifests that order of the radius of the ball in (iii) is generally the best possible.
Also, since $\psi\in C^3$, here one may have a direct treatment for (iii)
by a well-known approach to Inverse Function Theorem.
Indeed, set $F(x) = \nabla \psi(x)$, $|x|\leq c_1 a_1$
and let $y\in \R^d$ be fixed. We need to determine 
the range of $y$ where the equation $F(x) =y$ has a solution in $x$ from $B(0, c_1 a_1)$.
For this, one may utilize the celebrated method of Newton for finding roots of equations
by studying the mapping $G(x)=x- (\nabla F(0))^{-1} (F(x)-y) $, where the Jacobian
of $F$ is the Hessian of $\psi$. Clearly
$F(x) =y$ iff $G(x) = x$, i.e. it is enough to figure out when $G$ is a contraction.
The latter can be resolved easily relying on the $C^3$-smoothness of $\psi$,
and determining the range of $y$ when 
$G$ maps the closed ball $\overline{B}(0,c_0 a_1)$ into itself and has differential of norm less than 1.
The details are easy to recover and we omit them.
}.

\vspace{0.2cm}
The proof of the lemma is complete.
\end{proof}

\subsection{A prototypic domain}\label{sec-prot-domain}


We introduce a class of domains, call them \emph{prototypes}, which will be used
in the proof of Theorem \ref{Thm-slow-Dirichlet}.
Let $\mathcal{P}$ be a convex polytope, i.e. a convex bounded domain, which is an
intersection of a finite number of halfspaces. We assume that $\mathcal{P}\subset \{x\in \R^d: \ x_d\leq 0\}$
and that $0\in \R^d$ is an inner point of $\partial \mathcal{P} \cap \{x_d=0\}$.
We fix some $\Pi_0 \Subset \partial \mathcal{P} \cap \{x_d=0\}$, a $(d-1)$-dimensional closed ball
centred at 0.
Now let $D_0\subset \mathcal{P}$ be a bounded domain having the following properties:
\begin{itemize}
\item[{\normalfont(P1)}] $D_0$ is convex with $C^\infty$ boundary,
\vspace{0.1cm}
\item[{\normalfont(P2)}] $\partial D_0 \cap \{ x\in \R^d: \ x_d=0 \}=\Pi_0$, 
\vspace{0.1cm}
\item[{\normalfont(P3)}] if $x\in \partial D_0$ with $x_d\neq 0$, the Gaussian curvature of $\partial D_0$
at $x$ is strictly positive,
\vspace{0.1cm}
\item[{\normalfont(P4)}] we fix some ball $ B_0$ lying compactly inside $D_0$.
\end{itemize}

\vspace{0.1cm}

Typically we will embed the whole construction inside a given large domain $X$.
Existence of $D_0$ satisfying (P1)-(P4) follows directly, as a special case, from an elegant construction
due to M. Ghomi in connection with smoothing of convex polytopes, see \cite[Theorem 1.1]{Ghomi}.
The following picture gives a schematic view of the construction.

\begin{figure}[htb] 
\centering \def\svgwidth{300pt} 
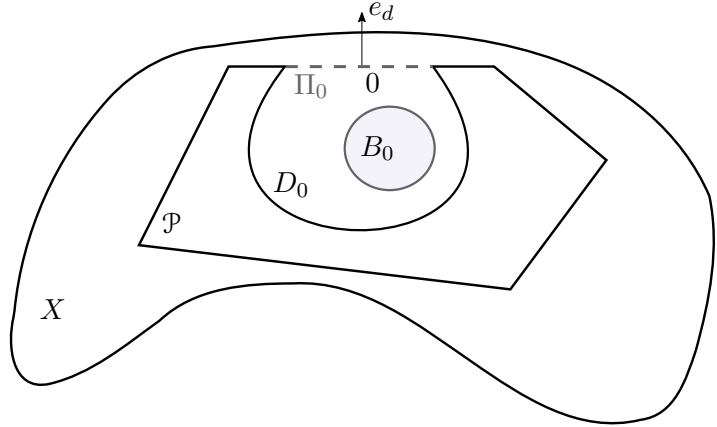 
\caption{\footnotesize{A prototypic domain $D_0$ obtained as smooth approximation of a polytope.
Here $X$ is some fixed domain containing $0\in \R^d$ in its interior.
Then $\mathcal{P}$ is any convex polygon sitting inside $X\cap\{x\in \R^d: x_d\leq 0  \}$,
with non-empty interior, and with part of its flat boundary lying on the hyperplane $\{x_d=0\}$.
We next take a closed flat ball $\Pi_0$, the dashed part on the boundary of $\mathcal{P}$,
and invoke \cite[Theorem 1.1]{Ghomi}. Finally, a ball $B_0$ is fixed compactly inside $D_0$.}}
\end{figure}

The following notation will be in force throughout the section.
Set $\Gamma_0 := \partial D_0$ and for $\delta>0$
denote $\Gamma_\delta = \{x\in \Gamma_0: \ \mathrm{dist}(x,\Pi_0) \geq \delta \}$,
where $\Pi_0$ is the $(d-1)$-dimensional ball fixed from (P2) above.
Define
\begin{equation}\label{def-of-kappa}
\kappa(\delta) =\min\limits_{x\in \Gamma_\delta} \min\limits_{1\leq \alpha \leq d-1} \kappa_\alpha(x),
\end{equation}
where $\kappa_\alpha(x)$ is the $\alpha$-th principal curvature of $\Gamma_0$ at $x$,
and the minimum over $\Gamma_\delta$ exists in view of the smoothness of $\Gamma_0$
and compactness of $\Gamma_\delta$. In the sequel we assume $\delta>0$ is small enough
so that $\Gamma_\delta \neq \emptyset$. Due to property (P3) we have
\begin{equation}
\kappa(\delta)>0  \text{ and } \kappa(\delta) \searrow 0  \text{ as } \delta \to 0+.
\end{equation}

The next proposition is one of the key ingredients of the proof of Theorem \ref{Thm-slow-Dirichlet}.

\begin{prop}\label{prop-main-for-Dirichlet}
There exists a modulus of continuity $\omega_0$ determined by the
decay rate of the function $\kappa(\delta)$ defined in (\ref{def-of-kappa})
such that 
for any smooth function $P: \Gamma_0 \to \R$, any $g\in C^\infty(\mathbb{T}^d)$
satisfying $\int_{\mathbb{T}^d} g =0$, any $y_0 \in \R^d$,  and any $M\in \OO(d)$ one has
\begin{equation}\label{int-in-prop}
\left| \int_{\Gamma_0 \setminus \Pi_0 } P(y) g(\lambda My +y_0) d\sigma(y) \right| \leq C 
\omega_0(\lambda) || P||_{C^1(\Gamma_0)} || g||_{C^d(\mathbb{T}^d)}, \qquad \forall \lambda \geq 1,
\end{equation} 
with a positive constant $C$ depending only on dimension $d$ and
embedding of $\Gamma_0$ in $\R^d$.
\end{prop}

\begin{remark}
One may claim a decay of the integral in (\ref{int-in-prop})
relying on \cite{Lee-Shah} for example, however without
any explicit bounds we have in the current formulation
and which we need for applications.
The proof of this proposition is based on
adaptation of methods from \cite{ASS1}-\cite{ASS2}
both of which work
with strictly convex domains, showing that integrals similar
to (\ref{int-in-prop}) and involving singular kernel (namely, Poisson's kernel) decay with some prescribed algebraic rate as $\lambda \to \infty$.
The difference of the current case from \cite{ASS1}-\cite{ASS2}
is that on one hand here we do not have a singularity introduced by an integration kernel,
which gives an extra freedom to the entire analysis.
On the other hand the strict convexity of the hypersurface deteriorates,
and we have integration over a hypersurface with boundary;
both of these factors introduce some technical difficulties
which entail somewhat refined analysis at certain points.
\end{remark}

\begin{proof}[Proof of Proposition \ref{prop-main-for-Dirichlet}]
The proof is partitioned into few steps.
\vspace{0.1cm}

\noindent \textbf{Step 1. Localization.} We localize the integral of (\ref{int-in-prop})
in a neighbourhood of each point $z\in \Gamma_0$ of positive curvature.
Fix $\delta>0$ small.
The hypersurface $\Gamma_0$ is locally a graph of a smooth function, thus
there exists $r>0$ small such that for any $z\in \Gamma_{2\delta}$ 
there is an orthogonal transformation $\mathcal{R}: \R^d \to \R^d$ satisfying
\begin{equation}\label{Gamma-is-loc-graph}
(\mathcal{R}(\Gamma_0 - z) ) \cap B_r(0) = \{(x', \psi(x')): \ |x'|\leq r \},
\end{equation}
where $x'=(x_1,...,x_{d-1})$, $\psi$ is smooth on $ \{x'\in \R^{d-1}: \ |x' | \leq 2r \}$,
$\psi(0)= |\nabla \psi(0)|=0$ and for the Hessian of $\psi$ we have
\begin{align}
\label{hessian-is-diag} \mathrm{Hess} \psi(0) =& \mathrm{diag} (a_1,...,a_{d-1}) \in M_{d-1}(\R), \\
\label{hessian-is-diag2} 0< \kappa(\delta) \leq& a_1\leq a_2\leq ...\leq a_{d-1} \leq C_0.
\end{align}
Here $a_\alpha$ is the $\alpha$-th
principal curvature of $\Gamma_0$ at $z$, and the lower bound of (\ref{hessian-is-diag2}) is due to (\ref{def-of-kappa}),
while the universal upper bound, as well as the bound $|\mathrm{Hess} \psi(x')|\leq C_0$,
for any $|x'|\leq 2r$, are both due to the smoothness of $\Gamma_0$.
Also, $r$ is independent of $z$ and $\delta$, and depends on the embedding of hypersurface $\Gamma_0$
in $\R^d$.

By Lemma \ref{Lem-Hessian-one-to-one} there exist constants
$K_1<K_2$ and $c_0<c_1$ controlled by dimension $d$ and $C^3$-norm of $\psi$,
and independent of the principal curvatures of $\Gamma_0$, such that
\vspace{0.1cm}
\begin{itemize}
\item[(a)] $ K_1 a_1 |x'| \leq |\nabla \psi(x') | \leq K_2 |x'|$ for all $|x'|\leq c_1 a_1$,
\vspace{0.1cm}
\item[(b)] for any $1\leq \alpha, \beta \leq d-1$
setting $\delta_{\alpha \beta}$ to be the Kronecker symbol, we
get
$$
\left| \partial_{\alpha \beta}^2 \psi (x') - a_\alpha \delta_{\alpha \beta} \right| \leq \frac{a_1}{10 d}, \qquad 
|x'|\leq c_1 a_1 ,
$$

\item[(c)] $|y'| < c_0 a_1^2 $ implies that there exists a unique $|x'| < c_1 a_1$ so that $\nabla \psi(x') =y'$.
\end{itemize}

\vspace{0.1cm}
Denote 
\begin{equation}\label{def-of-L-delta}
L_\delta = \frac{ c_0 }{4d K_2} a_1^2
\end{equation}
and consider a family of balls $\mathcal{B} =\{ B(z,\frac 15 L_\delta): \ z\in \Gamma_{2\delta}  \}$.
Clearly $\mathcal{B}$ covers $\Gamma_{2\delta}$, hence by Vitali covering lemma there exists
a finite collection of disjoint balls 
$\mathcal{B}_0  = \{ B(z_j, \frac 15 L_\delta): \ j=1,...,M_\delta  \}\subset \mathcal{B}$
such that 
\begin{equation}\label{gamma-delta-inclusion}
\Gamma_{2\delta} \subset \bigcup\limits_{j=1}^{M_\delta}  B(z_j, L_\delta) =: \widetilde{\Gamma}_{2\delta}.
\end{equation}
As $\Gamma_0$ is a graph in $L_\delta$-neighbourhood of any $z\in \Gamma_{2\delta}$
we get $ \mathrm{vol}_{d-1}( B(z_j,L_\delta/5) \cap \Gamma_{2\delta} ) \asymp L_\delta^{d-1} $.
From this, (\ref{def-of-L-delta}) and (\ref{hessian-is-diag2}) it easily follows that
\begin{equation}\label{est-on-M-delta}
M_\delta \leq C L_\delta^{-(d-1)}  \leq C (\kappa(\delta))^{-2(d-1)},
\end{equation}
with an absolute constant $C$.
For $1\leq j \leq M_\delta $ set $B_j: =B(z_j, L_\delta) $,
we now define a smooth partition of unity subordinate to these balls.
Fix a smooth function $\Psi:\R^d\to [0,1]$, such that
$\Psi(x) = 0$, for $|x|\geq 1$ and $\Psi(x)=1$ for $|x|\leq 1$.
Define $ \Psi_{\delta, j}(x) = \Psi(L_\delta^{-1}(x- z_j))$, then
clearly $\mathrm{supp} \Psi_{\delta, j} \subset B_j $.
Now let 
$ \varphi_{\delta, j}:=\left(  \sum\limits_{i=1}^{M_\delta} \Psi_{\delta, j} \right)^{-1} 
\Psi_{\delta,j}$.
By construction we have that $\varphi_{\delta, j}$ is supported in $B_j$,
$\sum\limits_{j=1}^{M_\delta} \varphi_{\delta, j} =1$ on $\Gamma_{2\delta}$, and
\begin{equation}
| \nabla \varphi_{\delta, j} (x) |\leq C L_\delta^{-1}, \qquad x\in \R^d,
\end{equation}
where the constant $C$ is independent of $j$ and $\delta$.
Now for $1\leq j \leq M_\delta$ define 
\begin{equation}\label{def-of-I-j}
I_j :=\int_{\Gamma_0} P(y)  g(\lambda My +y_0 ) \varphi_{\delta, j}(y) d\sigma(y).
\end{equation}

\vspace{0.2cm}

\noindent \textbf{Step 2. Reduction to oscillatory integrals.} In (\ref{def-of-I-j}) 
make a change of variables by setting $y=\mathcal{R}^{-1}z +z_j$,
where the orthogonal transformation $\mathcal{R}$ is fixed as in (\ref{Gamma-is-loc-graph}).
We next observe that the integral in (\ref{def-of-I-j}) is over $B(z_j, L_\delta)\cap \Gamma_0$, 
and $\Gamma_0$ is a graph in the $L_\delta$-neighbourhood of $z_j$.
With these in mind we make the change of variable in $I_j$ and then pass to volume integration
obtaining by so
\begin{equation}\label{I-j-2}
I_j=\int\limits_{|z'|\leq L_\delta}
(P\varphi_{\delta, j} )(z_j+\cR^{-1}(z',\psi(z')))  g \left(\lambda M z_j + \lambda M \cR^{-1} (z',\psi(z')) +y_0 \right)
(1+|\nabla \psi(z')|^2)^{1/2} d z' .
\end{equation}

For $t\in \R$ set $\Ex(t) := e^{2\pi i t}$.
Next, for $\xi\in \Z^d$ let $c_\xi(g)$ be the $\xi$-th Fourier coefficient of $g$.
According to the assumption of the proposition we have $c_0(g)=0$.
Using the smoothness of $g$ and expanding it into Fourier series
we get
\begin{multline}\label{exp of g}
g \left( \lambda M z_j + \lambda M \cR^{-1} (z',\psi(z')) +y_0 \right) = \\
\sum\limits_{\xi \in \Z^d \setminus\{0\} } c_\xi (g ) \Ex \left( \lambda \xi  \cdot M z_j +\xi \cdot y_0 \right)
\Ex \left[ \lambda  \cR M^T \xi \cdot  (z',\psi(z'))   \right],
\end{multline}
where we have also used the orthogonality of $M$ and $\cR$.
For $\xi\in \Z^d$ set $\cR M^T \xi := \eta := |\eta|(n', n_d)$
with $(n', n_d)\in \mathbb{S}^{d-1}$. By orthogonality of $M$ and $\cR$ we have $|\eta|=|\xi|$.
Next, define
$$
F(z') = n' \cdot z' + n_d \psi(z'),
$$
and
$$
\Phi_j(z') = (P\varphi_{\delta,j} )(z_j+\cR^{-1}(z',\psi(z')))  (1+|\nabla \psi(z')|^2)^{1/2} .
$$
From the definition of the cut-off $\varphi_{\delta,j}$ we have
\begin{equation}\label{C-1-norm-of-Phi}
|| \Phi_j||_{C^1} \leq C || P||_{C^1} L_\delta^{-1},
\end{equation}
uniformly for all $1\leq j \leq M_\delta$ with an absolute constant $C$ .
With these notation, from (\ref{I-j-2}) and (\ref{exp of g}) we get
\begin{equation}\label{I-j-as-sum}
I_j  = \sum_{\xi \in \Z^{d}\setminus \{0\} } c_\xi(g) \Ex \left( \lambda \xi  \cdot M z_j +\xi \cdot y_0  \right)
I_{j,\xi},
\end{equation}
where
\begin{equation}\label{I-j-xi}
I_{j,\xi} =\int_{|z'|\leq L_\delta} \Phi_j(z') \Ex \left[ \lambda |\xi| F(z')     \right] dz'.
\end{equation}

\noindent \textbf{Step 3. Decay of $I_j$.} 
We analyse the decay of each $I_{j,\xi}$
in two distinct cases.

\indent \textbf{Case 1.} $|n'|\geq c_0 a_1^2/2$.

Fix $1\leq \alpha \leq d-1$ so that $|n_\alpha| = \max_{1\leq \beta \leq d-1} |n_\beta|$,
clearly $|n_\alpha|\geq |n'|/d$.
From this, definition of $F$, (\ref{def-of-L-delta}) and assertion (a) of  Step 1, on the support of $\Phi_j$ we have
\begin{multline}\label{case1-est of phase from below}
|\partial_\alpha F(z')| = | n_\alpha + n_d \partial_\alpha \psi (z')| \geq |n_\alpha| -|\partial_\alpha \psi (z')| \geq \\
\frac{c_0 a_1^2}{2d} - K_2 L_\delta \geq \frac{c_0 a_1^2}{2d} - K_2 \frac{c_0 a_1^2}{4d K_2} =
\frac{c_0 }{4d} a_1^2.
\end{multline}
Integrating by parts in $I_{j,\xi}$ in the $\alpha$-th coordinate,
and then employing (\ref{case1-est of phase from below}) and (\ref{C-1-norm-of-Phi})
we get
$$
|I_{j,\xi}|\leq C (\lambda|\xi|)^{-1} \int_{|z'|\leq L_\delta}
\left| \partial_{\alpha} \left( \frac{\Phi_j}{\partial_\alpha F } (z') \right)  \right| dz'
\leq C (\lambda|\xi|)^{-1} L_\delta^{d-1} a_1^{-4} || P||_{C^1} L_\delta^{-1},
$$
with an absolute constant independent of $j$.
From here and (\ref{def-of-L-delta}) we have
\begin{equation}\label{I-j-xi-in-case1}
|I_{j,\xi}| \leq C (\lambda|\xi|)^{-1} ||P ||_{C^1} a_1^{2d-8} .
\end{equation}

\vspace{0.2cm}

\indent \textbf{Case 2.} $|n'|<c_0 a_1^2/2$.

Since $|(n',n_d)|=1$ and $c_0$ is small we have $|n_d|> 1/2$ and hence
$\left| \frac{n'}{n_d} \right|< c_0 a_1^2 $.
By (c) there exists a unique $x_0'\in B(0,c_1 a_1)$  such that $\nabla \psi (x_0') = -\frac{n'}{n_d}$,
and hence $\nabla F(x_0')=0$.
Observe that $x_0'$ is not necessarily from the support of $\Phi_j$.
For $1\leq \alpha \leq d-1$ consider the cone
$$
\mathcal{C}_\alpha=\left\{ z'\in \R^{d-1}: \ |z_\alpha |\geq \frac{1}{2 \sqrt{d-1} } |z'| \right \},
$$
clearly $\cup_{\alpha=1}^{d-1} \mathcal{C}_\alpha =\R^{d-1}$.
Now for fixed $1\leq \alpha \leq d-1$ take $z'\in \mathcal{C}_\alpha$
such that $|x_0'+z'|<c_1 a_1$.
Using estimate (b) of Step 1, the facts that $|n_d|>1/2$ and $\nabla F(x_0')=0$,
and invoking Mean-Value Theorem, for some $\tau$ on the segment $[x_0',x_0'+z'] $ we get
\begin{multline}\label{case2-est-of-deriv-F}
\left| \frac{\partial F}{\partial z_\alpha} (x_0' +z')  \right| = 
\left| \frac{\partial F}{\partial z_\alpha} (x_0' +z') - \frac{\partial F}{\partial z_\alpha} (x_0' ) \right | 
=\left| \left(\nabla \frac{\partial F}{\partial z_\alpha} \right) (\tau)\cdot z'  \right| \geq \\
|n_d| \left( |\partial^2_{\alpha \alpha}\psi(\tau) z_\alpha| -
\sum_{\beta =1, \ \beta\neq \alpha}^{d-1} |\partial^2_{\beta \alpha}\psi (\tau) z_\beta| \right) \geq C a_1 |z'|,
\end{multline}
with an absolute constant $C>0$.
Since the cones $\{C_\alpha\}$ cover $\R^{d-1}$,
for each $1\leq \alpha \leq d-1$ there exists $\omega_\alpha$ supported in $\mathcal{C}_\alpha$, smooth away
from the origin and homogeneous of degree 0
such that
$$
\sum\limits_{\alpha=1}^{d-1} \omega_\alpha (z') =1, \qquad \forall z'\neq 0.
$$
Observe that since each $\omega_\alpha$ is homogeneous of degree 0, for all $1\leq \alpha \leq d-1$
and non-zero $z'\in \R^{d-1}$ near the origin we have
\begin{equation}\label{deriv-of-omega}
\left|  \partial_\alpha \omega_\alpha  (z')  \right| \leq C \frac{1}{|z'|}, 
\end{equation}
with an absolute constant $C$.

Now fix a non-negative function $h\in C^{\infty}(\R^{d-1})$
satisfying $h(x')=0$ for $|x'|\geq 2$ and $h(x')=1$ for $|x'|\leq 1$.
Setting $x'=z'-x'_0$ form (\ref{I-j-xi}) we get
$$
I_{j,\xi}=\int_{\R^{d-1}} \Phi_j(x_0'+x') \exp \left[\lambda |\xi| F(x_0'+x') \right] dx' := I_{j,\xi}^{(1)} +I_{j,\xi}^{(2)},
$$
where
$$
I_{j,\xi}^{(1)} = \int_{\R^{d-1} } h( \lambda^{1/2} x' ) \Phi_j(x_0'+x')
\exp \left[ \lambda |\xi| F(x_0'+x') \right] dx' .
$$
From the definition of $h$ we have
\begin{equation}\label{est-of-I-xi-1}
|I_{j,\xi}^{(1)} | \leq C \lambda^{-(d-1)/2} || P||_{L^\infty}.
\end{equation}
The second part we decomposed as $I_{j,\xi}^{(2)}= \sum\limits_{\alpha=1}^{d-1} I_{j,\xi}^{(2),\alpha} $ where
$$
I_{j,\xi}^{(2),\alpha} =
\int_{\R^{d-1}} \omega_\alpha(x') (1-h(\lambda^{1/2} x')) \Phi_j(x_0'+x') \exp \left[ \lambda |\xi| F(x_0'+x') \right] dx' .
$$
We now invoke partial integration in $\alpha$-th coordinate,
and with the aid of estimates (\ref{case2-est-of-deriv-F}), (\ref{deriv-of-omega})
and (\ref{C-1-norm-of-Phi}) we get
\begin{multline*} 
I_{j,\xi}^{(2),\alpha}  \lesssim   
\frac{(\lambda |\xi|)^{-1} }{a_1} \int\limits_{|x_0' + x'|<L_\delta}
 \frac{1}{|x'|} \left[ \frac{1}{|x'|} \big( 1-h(\lambda^{1/2} x') \big) L_\delta^{-1} || P||_{C^1}
+ \lambda^{1/2}(\partial_{\alpha}h)(\lambda^{1/2} x') || P||_{L^\infty}   \right] dx' + \\  
\frac{(\lambda |\xi|)^{-1} }{a_1^2} \int\limits_{|x_0' + x'|<L_\delta} \frac{1}{|x'|^2} \big( 1-h(\lambda^{1/2} x') \big) || P||_{L^\infty} dx'.
\end{multline*}
We now use the definition of $h$,
and that of $L_\delta$ given by (\ref{def-of-L-delta}),
and employ integration in spherical coordinates by so appearing to
\begin{equation}\label{est-of-I-xi-2}
I_{j,\xi}^{(2),\alpha}  \lesssim 
\frac{(\lambda |\xi|)^{-1} }{a_1} \lambda^{1/2} L_\delta^{-1} || P ||_{C^1} + 
\frac{(\lambda |\xi|)^{-1} }{a_1^2} \lambda^{1/2} || P ||_{L^\infty} 
\end{equation}
with an absolute constant\footnote{One may get more precise decay rate in $\lambda$ depending on dimension,
cf. \cite[p. 76]{ASS2}, however the crude estimates we have here are enough for our purpose.} $C$.
By (\ref{est-of-I-xi-1}) and (\ref{est-of-I-xi-2}), along with the definition (\ref{def-of-L-delta}) we get
\begin{equation}\label{I-j-xi-in-case2}
|I_{j,\xi}|\leq C \lambda^{-1/2} || P||_{C^1} a_1^{-3}.
\end{equation} 

\vspace{0.2cm}

\noindent \textbf{Step 4. Final estimates.} 
We now put everything together.
By \cite[Lemma 2.3]{ASS1} we have
\begin{equation}\label{g-coeff-est}
\sum_{\xi\in \Z^d\setminus\{ 0 \} } |c_\xi(g)| \leq C \left(  \sum\limits_{\alpha \in \Z^d_+, \ |\alpha|=d}
 ||\nabla^\alpha g||_{L^2(\mathbb{T}^d)}^2 \right)^{1/2},
\end{equation}
where $\nabla^\alpha =\partial^{\alpha_1}_1 \circ ... \circ \partial^{\alpha_d}_d $,
$|\alpha|=|\alpha_1|+...+|\alpha_d|$, and constant $C=C(d)$.

Now let $I(\lambda)$ be the integral in (\ref{int-in-prop}).
From (\ref{def-of-I-j}) and (\ref{gamma-delta-inclusion}) we get
\begin{equation}\label{I-lambda-est1}
I(\lambda)=\sum\limits_{j=1}^{M_\delta} I_j + 
\int_{(\Gamma_0 \setminus \Pi_0)\setminus \widetilde{\Gamma}_{2 \delta}} P(y) g(\lambda M y + y_0) d \sigma(y).
\end{equation}
The definition of $\widetilde{\Gamma}_{2 \delta}$ infers
$\mathrm{vol}_{d-1}((\Gamma_0 \setminus \Pi_0)\setminus \widetilde{\Gamma}_{2 \delta}) \lesssim \delta$,
hence by (\ref{I-lambda-est1})
it follows that
\begin{equation}\label{I-lambda-est2}
|I(\lambda)|\lesssim \sum_{j=1}^{M_\delta}  |I_j| + \delta || P||_{L^\infty} ||g||_{L^\infty}.
\end{equation}
We thus get
\begin{align*}
  |I_j| &\leq \sum\limits_{\xi\neq 0} |c_\xi(g) | |I_{j,\xi}| \  \big( \text{by } (\ref{I-j-as-sum}) \big) \\
&\lesssim \sum\limits_{\xi\neq 0} |c_\xi(g) | 
\bigg[ (\lambda |\xi|)^{-1} ||P||_{C^1} a_1^{2d-8} + \lambda^{-1/2}||P||_{C^1} a_1^{-3} \bigg]
\big( \text{by } (\ref{I-j-xi-in-case1}) \text{ and }  (\ref{I-j-xi-in-case2}) \big) \\
&\lesssim \lambda^{-1/2} a_1^{-4} || P||_{C^1} \sum\limits_{\xi \neq 0} |c_\xi (g)| \\
&\lesssim \lambda^{-1/2} (\kappa(\delta))^{-4} || P||_{C^1} ||g||_{C^d} \ 
\big( \text{by } (\ref{g-coeff-est}) \text{ and } (\ref{hessian-is-diag2}) \big),
\end{align*}
where the constants are uniform in $1\leq j\leq M_\delta$ and $\delta>0$.
Using this bound on $I_j$ along with estimate (\ref{est-on-M-delta}) on $M_\delta$,
from (\ref{I-lambda-est2}) we get
\begin{equation}\label{I-lambda-last-est}
|I(\lambda)|\lesssim \big[\lambda^{-1/2} (\kappa(\delta))^{-2(d+1)} +\delta \big] || P||_{C^1} ||g||_{C^d},
\end{equation}
for any $\delta>0 $ small and any $\lambda \geq 1$. 
It is left to optimize the last
inequality in $\delta$, for which consider the function
$f(\delta) : = \delta (\kappa(\delta))^{2(d+1)} $
in the interval $(0,\delta_0)  $ where $\delta_0>0$ is small.
It follows from definition of $\kappa(\delta)$ in (\ref{def-of-kappa}) 
that $f$ is continuous and strictly increasing in $(0,\delta_0)$, and converges to $0$ as 
$\delta \to 0+$.
Hence for each $\lambda>0$ large enough there is a unique $0<\delta=\delta(\lambda)<\delta_0$ such that
$\lambda^{-1/2} = f(\delta(\lambda))$. Define $\omega_0 (\lambda) := \delta(\lambda)$
and observe that $\omega_0(\lambda) = f^{-1} (\lambda^{-1/2})$, where $f^{-1}$
is the inverse of $f$. It readily follows from 
the mentioned properties of $f$ that $\omega_0$ is a modulus of continuity.
Finally, for given $\lambda>0$ large, applying inequality 
(\ref{I-lambda-last-est}) with $\delta=\delta(\lambda)$
we get $|I(\lambda) | \lesssim \omega_0(\lambda) || P||_{C^1} ||g||_{C^d} $
completing the proof of the proposition.
\end{proof}

\begin{proof}[Proof of Theorem \ref{Thm-slow-Dirichlet}]
Recall that the coefficient matrix of (\ref{Dirichlet-bdd-domains}) is defined
on some fixed domain $X$. Take any $x_0$ in the interior of $X$
and consider the translated domain $X-x_0$.
Since the origin lies in the interior of $X-x_0$ we fix
$D_0 \Subset X-x_0$, any prototypic domain constructed in subsection \ref{sec-prot-domain}.
Following the notation of this section, we let
$\Gamma_0 := \partial D_0$ and by $\Pi_0 $ we denote the flat portion of $\Gamma_0$,
which by (P2) is a $(d-1)$-dimensional closed ball.
Finally, we fix a ball $B_0$ from property (P4) formulated above.

For a matrix $M\in \OO(d)$
let bounded domain $D_M$ be the rotated image of $D_0$ by $M$ as defined in subsection \ref{subsec-prelim}.
Set $\Gamma_M= \partial D_M$, as we have already discussed $\Gamma_M$ is smooth and
we have $\Gamma_M =\mathcal{M}(\Gamma_0) $.
Since $M$ is orthogonal, it follows that $\Pi_M: = \mathcal{M} \Pi_0 \subset \partial D_M$
is a $(d-1)$-dimensional ball lying in a hyperplane passing through the origin and having normal
equal to $M e_d$.
Finally, we fix a small closed neighbourhood $S_0\subset \mathbb{S}^{d-1}$ of $e_d$ 
with non-empty interior, such that
for any $n\in S_0$ if $M\in \OO(d)$ with $M e_d = n$, then the corresponding rotated domain $D_M$
lies in $X-x_0$ and compactly contains the ball $B_0$.

Fix a function $g\in C^\infty(\mathbb{T}^d)$ with the property $\int_{\mathbb{T}^d} g =0$,
and a unit vector $n\in S_0$ along with a matrix $M \in \OO(d)$ satisfying $M e_d = n$.
For $\e>0$ let $u_\e$ be the smooth solution to
\begin{equation}\label{osc-prob-in-proof} 
 -\nabla \cdot A(x) \nabla u_\e(x) = 0 \text{ in } D_M +x_0 \qquad \text{ and } \qquad u_\e(x) = g (x/ \e) \text{ on } \Gamma_M+x_0.
\end{equation} 
Since $g$ has mean zero, by (\ref{Dirichlet-bdd-domains-homo}) we get that $u_0$, the homogenized
solution corresponding to (\ref{osc-prob-in-proof}), is identically zero.
Now let $P_M: (D_M+x_0) \times (\Gamma_M+x_0) \to \R$ be the Poisson kernel for the pair $(A, D_M+x_0)$.
For $x\in D_M+x_0$ we have
\begin{equation}\label{u-e-two-parts}
u_\e (x) = \int_{\Gamma_M+x_0} P_M(x,y) g(y/ \e) d\sigma_M(y) = \int_{\Pi_M+x_0} + \int_{\Gamma_M \setminus \Pi_M+x_0} 
:= I_1 + I_2.
\end{equation}
Observe that $\Gamma_M \setminus \Pi_M$ is the part of the boundary of $D_M$ with non-vanishing (positive) Gaussian curvature,
while $\Pi_M$ is a flat ball (notice that principal curvatures are orthogonal invariants, hence do not change after rotation,
see for example \cite[Section 14.6]{GT}). The aim is to show that $I_2$
has some prescribed decay rate in $\e$ independently of $M$,
while the decay of $I_1$ can be made as slow as one wish.
Setting $\lambda = 1/\e$ and then translating the origin to $x_0$ and rotating the coordinate system by $M^T$ we get
$$
I_2=I_2(\lambda; x)=
 \int_{\Gamma_0 \setminus \Pi_0} P_M(x,M z+x_0) g( \lambda M z + \lambda x_0) d\sigma_0(z) .
$$
For each fixed $x\in B_0+x_0$ by (\ref{Poisson-rep}) the function $P_M(x, \cdot +x_0 )$ is $C^\infty$ on $\Gamma_M$, hence applying
Proposition \ref{prop-main-for-Dirichlet} we obtain
$$
|I_2| \leq C_0 \omega_0(\lambda ) ||P_M (x, \cdot +x_0) ||_{C^1(\Gamma_M)} ||g||_{C^d(\mathbb{T}^d)}, \qquad x\in B_0+x_0,
$$
with an absolute constant\footnote{Observe that the constant in Proposition \ref{prop-main-for-Dirichlet}
is independent of the shift and hence having $\lambda x_0$ in $g$ still results
in a constant independent of $\lambda$.}
$C_0$. By Lemma \ref{Lem-Poisson-dist-smooth} we have
$ ||P_M (x, \cdot + x_0) ||_{C^1(\Gamma_M)} \leq C_0 $ uniformly for $x\in B_0 +x_0$ and matrix $M$
as above.
The latter implies that
\begin{equation}\label{decay-of-I-2}
|I_2(\lambda; x)| \leq C_0 \omega_0(\lambda ) ||g||_{C^d (\mathbb{T}^d) }, \qquad  x\in B_0+x_0.
\end{equation}

For the flat part of the integral we denote by $\mathbb{I}_{M}$ the characteristic
function of $\Pi_M$ and extend $P_M(x, \cdot )$ as zero outside $\Pi_M+x_0$.
With these notation we get
\begin{align*}
\numberthis \label{I-1} I_1= I_1(\lambda; x) = \int_{\Pi_M} P_M(x,y +x_0) g(\lambda y + \lambda x_0) d\sigma_M(y) &=  \\
 \int_{y\cdot n =0} P_M(x,y +x_0)\mathbb{I}_{M}(y) g(\lambda y + \lambda x_0) d \sigma(y) &= \\
 \int_{z_d= 0} P_M(x , Mz +x_0 ) \mathbb{I}_{M}(Mz  ) g(\lambda Mz +\lambda x_0 ) d\sigma_0(z) &= \\
 \int_{\R^{d-1}} P_M \left(x , M (z',0)+  x_0 \right) 
 \mathbb{I}_{M} \left(M (z',0)   \right) g \left(\lambda M(z',0)+ \lambda x_0 \right) d z' &.
\end{align*}

Consider the following 2-parameter family of functions
\begin{multline}
\mathcal{F}: = \{  F_{x,M}(z') : = P_M(x , M (z',0) +  x_0  ) \mathbb{I}_{M} (M (z',0) ), \  z'\in \R^{d-1}: \text{ where }  \\
\ \ x\in B_0 +x_0 , \ M\in \OO(d) \text{ with } M e_d =n \text{ and } n\in S_0   \}.
\end{multline}
Let us see that the family $\mathcal{F}$ satisfies all conditions of Lemma \ref{Lem-family-on-compact-supp}.
Indeed, the assumption (c') of the lemma is simply due to the construction of domain $D_0$
and orthogonality of matrices $M$.
Next, (c') combined with the estimate (\ref{dist-Poisson-smooth}) implies assumption (b)
of Lemma \ref{Lem-slow-conv-family-of-F}.
Items (d) and (e) of Lemma \ref{Lem-family-on-compact-supp} are due to 
(\ref{deriv-Poisson-smooth}) and (\ref{dist-Poisson-smooth}) correspondingly.
It is left to verify the assumption with the lower bound on the integrals
in part (a) of Lemma \ref{Lem-slow-conv-family-of-F}. The latter is due to
Corollary \ref{cor-Poisson-inf-rotated}.

We now apply Lemma \ref{Lem-family-on-compact-supp} to the family $\mathcal{F}$ for modulus of continuity $\lambda\longmapsto \omega(\lambda) + \omega_0^{1/2}(\lambda)$,
where $\omega$ is the one fixed in the formulation of the theorem, and
$\omega_0$ is determined from (\ref{decay-of-I-2}).
By doing so we get
an irrational normal $n\in S_0$, a function $g \in C^\infty(\mathbb{T}^d)$
with the property $\int_{\mathbb{T}^d} g =0$,
and a strictly increasing sequence of positive numbers $\{\lambda_k\}_{k=1}^\infty$
such that for any $M\in \OO(d)$ satisfying $M e_d=n$ from (\ref{I-1}) we have
\begin{equation}\label{I-1-on-lambda-k}
|I_1(\lambda_k; x)| \geq \omega(\lambda_k) + \omega_0^{1/2}(\lambda_k), \ \   x\in B_0+x_0,  \qquad k=1,2,... \ .
\end{equation}

Since $\omega_0$ decays at infinity, combining (\ref{I-1-on-lambda-k}) and (\ref{decay-of-I-2})
and setting $\e_k = 1/\lambda_k$, 
from the representation (\ref{u-e-two-parts})
for any $x\in B_0 +x_0$ we get
\begin{equation}\label{u-e-k}
|u_{\e_k}(x)| \geq |I_1(\e_k; x)|-|I_2(\e_k;x)| \geq \omega(\lambda_k) + \omega_0^{1/2}(\lambda_k) -
C_0 \omega_0(\lambda_k) ||g||_{C^d(\mathbb{T}^d)} \geq \omega(1/\e_k) ,
\end{equation}
if $k\geq 1$ large enough. 
Finally, as a domain $D$ of the Theorem we take $D_M+x_0$ where $D_M$ is any domain
having $n$ as the normal of its flat boundary, where $n$ is obtained from Lemma \ref{Lem-family-on-compact-supp} (note that $D$ is defined modulo $\OO(d-1)$), and as $D'$ we take $B_0+x_0$, where $B_0$ is the ball fixed above.
Since $\int_{\mathbb{T}^d} g =0$,
the homogenized problem (\ref{Dirichlet-bdd-domains-homo}) has only a trivial solution.
Thus (\ref{u-e-k}) completes the proof of part a) of the theorem.

\vspace{0.2cm}

We now prove part b). Observe that an argument similar to what we had for
(\ref{decay-of-I-2}) shows that $I_2(\lambda, x) \to 0 $ as $\lambda \to \infty$
for any $x\in D=D_M+x_0$. Hence, taking into account
decomposition (\ref{u-e-two-parts}), 
it is enough to show that $I_1(\lambda; x) \to 0$ as $\lambda \to \infty$
for all $x\in D$.
We may assume without loss of generality, by passing to a subset of $S_0$
if necessary, that $n_d\neq 0$, where the unit vector $n$ is fixed from part a).
Now take any subset $\Pi$ of $\Pi_M+x_0$ such that the projection of $\Pi$
onto $\R^{d-1}\times\{0\} $ is a $(d-1)$-dimensional rectangle,
which we will denote by $[a_1,b_1]\times ... \times [a_{d-1}, b_{d-1}]$.
To show that $I_1(\lambda; x) $ decays it is enough to prove that
the integral
\begin{equation}\label{I-lambda-over-Pi}
I(\lambda; x):= \int_{\Pi } P_M(x,y) g(\lambda y) d\sigma_M(y) 
\end{equation}
converges to 0 as $\lambda \to \infty$ for any $x\in D$. 
The latter follows easily.
Indeed, since $n_d\neq 0$ for $y\in \Pi$ we have
$y_d = -\frac{1}{n_d} ( n_1 y_1 +...+n_{d-1} y_{d-1} ) $ hence passing to volume integration
in (\ref{I-lambda-over-Pi}) and expanding $g$ into Fourier series,
for each $x\in D$ we get
\begin{equation}\label{I-lambda-over-Pi-rect}
I(\lambda; x) = \sum_{\xi \neq 0} c_\xi(g) 
\int_{a_1}^{b_1}...\int_{a_{d-1}}^{b_{d-1}} P_M(x, y_1,..., y_d )
\prod\limits_{k=1}^{d-1}  \Ex \left[ \left( \xi_k - \frac{n_k}{n_d} \xi_d  \right)  \lambda y_k \right]
d y_1... d y_{d-1},
\end{equation}
where $y_d =-\frac{1}{n_d} ( n_1 y_1 +...+n_{d-1} y_{d-1} )$,
and as before $\Ex(t) = e^{2\pi i t}$, $t\in \R$.
Since $n$ is irrational, it is easy to see that at least one of the exponentials
in the last expression is non-trivial, i.e. for any non-zero $\xi\in \Z^d $
 there is $1\leq k \leq d-1$ such that $ \xi_k - \frac{n_k}{n_d} \xi_d \neq 0$.
Using this and the smoothness of $P_M(x, \cdot)$ which is due to Lemma \ref{Lem-Poisson-dist-smooth},
in (\ref{I-lambda-over-Pi-rect}) for each non-zero $\xi$ we integrate by parts with respect to the corresponding $k$-th
 coordinate, obtaining by so that
 each integral in (\ref{I-lambda-over-Pi-rect}) decays
 as $\lambda \to \infty$. Observe as well that for each fixed $x$, integrals in (\ref{I-lambda-over-Pi-rect})
are uniformly bounded with respect to $\lambda$ and $\xi$ in view of (\ref{dist-Poisson-smooth}).
Finally, the series of Fourier coefficients of $g$ converges absolutely
due to the smoothness of $g$.
This, coupled with uniform boundedness and decay of integrals in (\ref{I-lambda-over-Pi-rect}),
easily implies that $I(\lambda; x)\to 0 $ as $\lambda \to \infty$ for each fixed $x\in D$,
completing the proof of part b) of the theorem.

The theorem is proved.
\end{proof}

\subsection{Concluding comments}
The reader may have observed that the approach we have here
has a potential to work for homogenization problems
when solutions to the underlying PDE admit integral representation
with some nice control for representation kernels.
For example, one should be able to treat the periodic homogenization
of Neumann boundary data with the methods developed in this note.
It also seems plausible, possibly with some more work, that one can study homogenization
of almost-periodic boundary data as well using a similar analysis.

The reason we can not readily allow (periodic) oscillations in the operator 
of the problem (\ref{Dirichlet-bdd-domains}) is due to the absence of
the necessary control over the Poisson kernel
corresponding to oscillating operator. In particular, we do not have uniform (with respect to $\e>0$)
bounds on $C^1$-norms similar to those in Lemma \ref{Lem-Poisson-dist-smooth}
at our disposal. One possible detour of this obstacle 
is the use of results on homogenization of Poisson kernel for the $\e$-problem
considered in \cite{Kenig-Lin-Shen-CPAM}, to reduce matters to a fixed operator,
where the analysis of the current paper can be utilized (cf. the proof of Theorem 1.7 of \cite{ASS2}).

\section*{Appendix}

We give some basic estimates on Green's and Poisson's kernels associated with divergence type elliptic operators in bounded domains. The estimates we will need are standard and well-known,
but in view of the absence of an explicit reference we include the proofs here.

Following subsection \ref{subsec-prelim}
for a matrix $M\in \OO(d)$ and domain $D\subset \R^d$
by $D_M$ we denote the rotated image of $D$ by the matrix $M$.
We saw already that
$D_M$ is a bounded domain with $C^\infty$ boundary, moreover
$\partial D_M =\mathcal{M} (\partial D)$, where the orthogonal transformation $\mathcal{M}$
is defined through $\mathcal{M}x = Mx$ for $x\in \R^d$.
Next, we let the coefficient matrix $A$ and domain $X$ be as in the formulation of Theorem \ref{Thm-slow-Dirichlet}.
Finally, we fix a bounded domain $D$ with $C^\infty $ boundary,
and such that for any $M\in \OO(d)$ the domain $D_M$ lies compactly inside $X$.

With these preliminary setup we now consider the operator
$\mathcal{L}:= - \nabla \cdot A(x) \nabla $, $x\in X$,
and let $G_M(x,y)$ be the Green's kernel for $\mathcal{L}$ in domain $D_M$, i.e.
for each fixed $y \in D_M$, $G_M(\cdot , y)$ solves 
$$
-\nabla_x \cdot  A(x) \nabla_x G_M(x,y) = \delta(x-y), \ \ x\in D_M
\qquad \text{ and } \qquad G_M(x,y)=0 , \ x\in \partial D_M,
$$
in a sense of distributions, where $\delta$ is the Dirac symbol.
We have the following properties.

\begin{itemize}
\indentdisplays{1pt}

 \item[(a)] Set $G_M^*(x,y) := G_M^T(y,x)$, then $G_M^*$ is the Green's kernel for the formal adjoint to $\mathcal{L}$
 in domain $D_M$,  i.e. a divergence type operator with coefficient matrix equal to $A^{\beta \alpha}$.

\vspace{0.1cm}
 
 \item[(b)] For any multi-index $\alpha \in \Z^d_+$ we have
 \begin{equation}\label{sDom-Green-deriv}
  |\nabla^\alpha_x G_M (x,y)| \leq C_\alpha |x-y|^{2-d -|m|} \text{ for } d+|m|>2,
 \end{equation}
 where $C_\alpha $ is independent of $M$.
\end{itemize}

\vspace{0.1cm}

The existence and uniqueness of Green's kernel, as well as property (a) and 
estimate of (b) for each fixed $M$ are proved
in \cite{Dolz-Muller}, which is one of the several papers discussing
some basic questions around the Green's kernels for divergence type elliptic operators.
The only thing that needs some clarification is the choice of the constant in (b)
independently of $M$. To see it, we proceed as follows.  

Assume $G$ is the Green's kernel for the pair ($A, D$) and let $M\in \OO(d)$ be fixed.
Make a change a variables by setting
$Mx = \widetilde{x}$ and $My = \widetilde{y}$ where $x,y\in D$ and $\widetilde{x}, \widetilde{y} \in D_M$.
By orthogonality of $M$ we have $\nabla_x  = M^T \nabla_{\widetilde{x}}$ which
combined with the uniqueness of the Green's kernel, easily implies that the function 
$$
\widetilde{G} (\widetilde{x}, \widetilde{y}):= G(M^T \widetilde{x}, M^T \widetilde{y})
$$
is the Green's kernel for the pair $( M A(M^T \cdot ) M^T , D_M )$.
This shows that if $\widetilde{G}_M$ is the Green's kernel for the pair $(M^T A(M \cdot )M, D)$,
then 
\begin{equation}\label{rel-between-greens}
G_M (\widetilde{x}, \widetilde{y}  ) := \widetilde{G}_M (M^T \widetilde{x}, M^T \widetilde{y} ) 
\end{equation}
produces the Green's kernel for the pair $(A, D_M)$.
In view of the orthogonality of $M$ the ellipticity constants as well as bounds on $C^k$
norms, for any $k\geq 0$, of the coefficient matrix $M^T A(M \cdot) M$ can be chosen independently of $M$,
hence (\ref{rel-between-greens}) and the proof of \cite{Dolz-Muller} illustrate
uniformity of the constants in (b) with respect to $M$.

Now for $x \in D_M$ and $y \in \partial D_M$ we let $P_M(x,y)$ be the Poisson kernel
for the pair $(A, D_M)$. One may easily conclude using the divergence theorem that 
\begin{equation}\label{Poisson-rep}
P_M(x,y) = - n^T(y) A^T(y) \nabla_y G_M(x,y), \ \  \text{ where } x \in D_M, \ y \in \partial D_M,
\end{equation}
where $n(y)$ is the unit outward normal to $\partial D_M$ at the point $y$.

\begin{lem}\label{Lem-Poisson-dist-smooth}
Let the domain $D$ be as above and $M\in \OO(d)$ be any.
Keeping the above notation and assumptions in force we have
\begin{equation}\label{dist-Poisson-smooth}
|P_M(x,y)| \leq C \frac{d(x)}{|x-y|^d}, \qquad \forall x\in D_M, \ \forall y\in \partial D_M,
\end{equation}
where $d(x)$ is the distance of $x$ from the boundary of $D$, and
\begin{equation}\label{deriv-Poisson-smooth}
 |\nabla_y P_M(x,y)| \leq C |x-y|^{-d} , \qquad \forall x\in D_M, \ \forall y\in\partial D_M.
\end{equation}
with constants independent of $M$.
\end{lem}

\begin{proof}
Let us first observe that the estimates on derivatives of $P$ follow trivially from
the representation (\ref{Poisson-rep}), estimate 
(\ref{sDom-Green-deriv}), the symmetry property of Green's matrix (a) formulated above,
combined with smoothness of the coefficients $A$ and domain $D_M$. 

We now proceed to the proof of the estimate with distance,
for which we will rely on uniformity of constants in (\ref{sDom-Green-deriv})
with respect to $M$.
First consider the case 
of $d\geq 3$. We now show that for $x,y\in D$ with $x\neq y$ one has
\begin{equation}\label{sDom-Green-dist1}
| G_M(x,y) | \leq C \frac{d(x)}{ |x-y|^{d-1} }.
\end{equation}
By (\ref{sDom-Green-deriv}) we have $|G_M(x,y)|\leq C |x-y|^{2-d}$, hence (\ref{sDom-Green-dist1})
is trivial if $d(x)>\frac 13 |x-y|$, thus we will assume that $d(x)\leq \frac 13 |x-y|$.
Now fix $\overline{ x } \in \partial D$ such that $d(x) = |x-\overline{x}|$.
Since $G_M$ vanishes on the boundary of $D_M$ with respect to both variables, using Mean-Value Theorem
and estimate (\ref{sDom-Green-deriv}) for $G_M^*$ we get
\begin{equation}\label{Z1}
|G_M(x,y)| = |G_M(x,y) - G_M( \overline{x}, y) | \leq | \nabla_x G_M( \widetilde{x}, y) |   | x-\overline{x}| \leq C \frac{d(x)}{| \widetilde{x} -y |^{d-1} }.
\end{equation}
Here, for estimating the derivative of $G_M$ we have used symmetry relation (a) above.
As $d(x) \leq \frac 13 |x-y|$, and $\widetilde{x}$ is on the segment $[x, \overline{x}]$,
the triangle inequality implies
$$
| \widetilde{x} - y | \geq |x-y|-|x- \widetilde{x} | \geq |x-y| - |x - \overline{x}| \geq  \frac 23 |x-y|,
$$
which combined with (\ref{Z1}) gives (\ref{sDom-Green-dist1}). 

Now fix $x_0,y_0$ and let $r:=|x_0-y_0|>0$.
Consider $\widetilde{G}_M(z): = G_M(x_0, rz + x_0 )$ in a scaled and shifted domain
$D_{r,M}:= r^{-1}(D_M-x_0)$. We have that $\widetilde{G}_M$
is a solution to the adjoint operator in $D_{r,M} \cap (B(0, 3)\setminus B( 0, 1/3))$, hence in view of the smoothness
of the coefficients, from standard elliptic regularity
estimates we obtain
$$
|\nabla_z \widetilde{G}_M(z)| = | r \nabla_y G_M (x_0, rz + x_0 ) | \leq C ||  \widetilde{G}_M ||_{L^\infty( D_{r,M} \cap( B(0,3)\setminus B(0,1/3) ))},
$$
for all $z\in D_{r,M} \cap (B(0,2)\setminus B(0,1/2)) $.
Note that the regularity estimates we have used are uniform in $M$
due to (\ref{rel-between-greens}) and orthogonality of $M$. The last inequality combined with (\ref{sDom-Green-dist1}) implies
\begin{equation}\label{Z2}
| \nabla_y G_M(x_0,y_0)| \leq C \frac{d(x_0)}{| x_0 -y_0 |^d}.
\end{equation}
Now the estimate (\ref{dist-Poisson-smooth}) follows from the last inequality and
representation (\ref{Poisson-rep}).

It is left to consider the 2-dimensional case, which can be handled precisely as in \cite[Lemma 21]{AL-systems}
thus we will omit the details.

The lemma is proved.
\end{proof}

\noindent {\footnotesize \textbf{Acknowledgements.}
I thank Christophe Prange for an important conversation
back in Fall of 2014 which has eventually led to the problem considered in Section \ref{sec-Laplace}.
Part of the work was done
during my visit to Institut Mittag-Leffler for the term
``Homogenization and Random Phenomenon". The warm hospitality and support 
of The Institute is acknowledged with gratitude.
This article is partially based on my PhD thesis completed at The University of Edinburgh in 2015,
and I wish to thank my thesis supervisor Aram Karakhanyan for
his encouragement, as well as advice.
I am also grateful to the anonymous referees for thorough reading of the manuscript
and providing valuable suggestions, as well as corrections which have certainly helped to
improve the quality of the presentation.
}

\end{document}

%% file: fig1.pdf_tex
\begingroup%
  \makeatletter%
  \providecommand\color[2][]{%
    \errmessage{(Inkscape) Color is used for the text in Inkscape, but the package 'color.sty' is not loaded}%
    \renewcommand\color[2][]{}%
  }%
  \providecommand\transparent[1]{%
    \errmessage{(Inkscape) Transparency is used (non-zero) for the text in Inkscape, but the package 'transparent.sty' is not loaded}%
    \renewcommand\transparent[1]{}%
  }%
  \providecommand\rotatebox[2]{#2}%
  \ifx\svgwidth\undefined%
    \setlength{\unitlength}{844.07558594bp}%
    \ifx\svgscale\undefined%
      \relax%
    \else%
      \setlength{\unitlength}{\unitlength * \real{\svgscale}}%
    \fi%
  \else%
    \setlength{\unitlength}{\svgwidth}%
  \fi%
  \global\let\svgwidth\undefined%
  \global\let\svgscale\undefined%
  \makeatother%
  \begin{picture}(1,0.54213154)%
    \put(0,0){\includegraphics[width=\unitlength,page=1]{fig1.pdf}}%
    \put(0.48770177,0.37098912){\color[rgb]{0,0,0}\makebox(0,0)[lt]{\begin{minipage}{0.07988452\unitlength}\raggedright $B_0$\end{minipage}}}%
    \put(0.377759,0.32549555){\color[rgb]{0,0,0}\makebox(0,0)[lt]{\begin{minipage}{0.08936234\unitlength}\raggedright $D_0$\end{minipage}}}%
    \put(0.23775799,0.27512767){\color[rgb]{0,0,0}\makebox(0,0)[lt]{\begin{minipage}{0.2965205\unitlength}\raggedright $\mathcal{P}$\end{minipage}}}%
    \put(0.08394645,0.16572651){\color[rgb]{0,0,0}\makebox(0,0)[lt]{\begin{minipage}{0.12185774\unitlength}\raggedright $X$\end{minipage}}}%
    \put(0.49474242,0.45114443){\color[rgb]{0,0,0}\makebox(0,0)[lt]{\begin{minipage}{0.11779581\unitlength}\raggedright $0$\end{minipage}}}%
    \put(0,0){\includegraphics[width=\unitlength,page=2]{fig1.pdf}}%
    \put(0.11373389,0.72356418){\color[rgb]{0,0,0}\makebox(0,0)[lt]{\begin{minipage}{0.13268954\unitlength}\raggedright \end{minipage}}}%
    \put(0.49826275,0.53942359){\color[rgb]{0,0,0}\makebox(0,0)[lt]{\begin{minipage}{0.18684853\unitlength}\raggedright $e_d$\end{minipage}}}%
    \put(0.25725522,0.85963865){\color[rgb]{0,0,0}\makebox(0,0)[lt]{\begin{minipage}{0.17872468\unitlength}\raggedright \end{minipage}}}%
    \put(0.26470208,0.66209372){\color[rgb]{0,0,0}\makebox(0,0)[lb]{\smash{}}}%
    \put(0.40456769,0.44979045){\color[rgb]{0.39215686,0.39215686,0.39215686}\makebox(0,0)[lt]{\begin{minipage}{0.11711883\unitlength}\raggedright $\Pi_0$\end{minipage}}}%
  \end{picture}%
\endgroup%